\documentclass[11pt]{amsart}

\usepackage{amssymb}
\usepackage{latexsym}
\usepackage{graphicx}

\theoremstyle{plain}
\newtheorem{thm}{Theorem}[section]

\newtheorem{prop}[thm]{Proposition}
\newtheorem{remark}[thm]{Remark}
\newtheorem{cor}[thm]{Corollary}
\newtheorem{lemma}[thm]{Lemma}

\newtheorem{prob}[thm]{Problem}

\newtheorem{ex}[thm]{Example}

\usepackage{color}

\theoremstyle{definition}
\newtheorem{defn}[thm]{Definition}

\newtheorem{rem}[thm]{Remark}

\numberwithin{equation}{section}

\renewcommand{\leq}{\leqslant}
\renewcommand{\geq}{\geqslant}

\begin{document}

\title[More on measurable algebras and Rademacher systems]{More on measurable algebras and Rademacher systems with applications to analysis of Riesz spaces}

\author[M.~Popov]{Mikhail Popov}

\address{Institute of Mathematics\\
Pomeranian University in S{\l}upsk\\
Arciszewskiego 22d\\
76-200 S{\l}upsk (Poland)}

\address{and Department of Mathematics and Informatics\\
Chernivtsi National University\\
str.~Kotsyubyns'koho 2, Chernivtsi, 58012 (Ukraine)}


\email{misham.popov@gmail.com}

\subjclass[2010]{Primary 28A60; secondary 46G12}

\keywords{Boolean algebra, measurable algebra, Rademacher system, Haar system, Riesz space, order Schauder basis}

\begin{abstract}
We find necessary and sufficient conditions on a family $\mathcal{R} = (r_i)_{i \in I}$ in a Boolean algebra $\mathcal{B}$ under which there exists a unique positive probability measure $\mu$ on $\mathcal{B}$ such that $\mu ( \bigcap_{k=1}^n \theta_k r_{i_k} ) = 2^{-n}$
for all finite collections of distinct indices $i_1, \ldots, i_n \in I$ and all collections of signs $\theta_1, \ldots, \theta_n \in \{-1,1\}$, where the product $\theta x$ of a sign $\theta$ by an element $x \in \mathcal{B}$ is defined by setting $1 x = x$ and $-1 x = - x = \mathbf{1} \setminus x$. Such a family we call a complete Rademacher family. We prove that Dedekind $\sigma$-complete Boolean algebras admitting complete Rademacher systems of the same cardinality are isomorphic. As a consequence, we obtain that a Dedekind $\sigma$-complete Boolean algebra is homogeneous measurable if and only if it admits a complete Rademacher family. This new way to define a measure on a Boolean algebra allows us to define classical systems on an arbitrary Riesz space, such as Rademacher and Haar. We define a complete Rademacher system of any cardinality and a countable complete Haar system on an element $e > 0$ of a vector lattice $E$ in such a way that if $e$ is an order unit of $E$ then the corresponding systems become complete for the entire $E$. We prove that if $E$ is Dedekind complete then any complete Haar system on $e$ is an order Schauder basis for the ideal $A_e$ generated by $e$. Finally, we develop a theory of integration in a Riesz space of elements of the band $B_e$ generated by a fixed $e > 0$ with respect to the measure on the Boolean algebra $\mathfrak{F}_e$ of fragments of $e$ generated by a complete Rademacher family on $\mathfrak{F}_e$. Much space is devoted to examples showing that our way of thinking is sharp (e.g., we show the essentiality of each of the condition in the definition of a Rademacher family).
\end{abstract}

\maketitle

\tableofcontents

\section{Introduction}
\label{sec1}

$\,\,\,\,\,\,$\textbf{The idea.} The investigation of the paper was partially inspirit by the following problem.

\begin{prob} \label{pr:main}
Let $\mathcal{F} = (f_n)_{n=1}^\infty$ be a sequence of real-valued functions defined on an abstract set $X$ (or, more general, a sequence in a Riesz space). Under what conditions, $\mathcal{F}$ can be considered as a sequence of independent random variables? As a Rademacher system? As a sequence of independent Gaussian variables? More precisely, under what conditions on $\mathcal{F}$ there exist a probability space $(\Omega, \Sigma, \mu)$, a Riesz space $F$ of functions $f: X \to \mathbb R$ containing $\mathcal{F}$ and a Riesz into isomorphism $T: F \to L_0(\mu)$ such that $(T f_n)_{n=1}^\infty$ is a sequence of independent random variables (or even with additional properties)?
\end{prob}

In particular, we answer this problem in its part concerning Rademacher systems.

The classical Rademacher system is an important tool in the investigation of the isomorphic structure of symmetric (rearrangement invariant) spaces \cite{As}, K\"{o}the function spaces on measure spaces \cite{PR}, as well as in probability theory. So, it would be natural to generalize a Rademacher system to the setting of Riesz spaces. However, the definition of a Rademacher type system uses a measure which is not generally defined if we consider an arbitrary Riesz space. Our main idea is to define a Rademacher type system on a general Riesz space without a measure. Moreover, a Rademacher system, which is defined axiomatically, generates a probability measure by an obvious way. Since any element of a Rademacher type system in a Riesz space $E$ has to be of the form $r = a - b$, where $a,b \in E^+$, $a \perp b$ and the element $e = |r| = a \sqcup b$ (by $a \sqcup b$ we denote a disjoint sum, that is the sum $a+b$ of disjoint elements $a \perp b = 0$) is some fixed element of $E$ considered as the ``support'' of the Rademacher system, to define a Rademacher system on $e$ means to find a sequence of two-element partitions of $e$. Observe that elements $a,b$ of any partition $e = a \sqcup b$ are fragments of $e$, that is, $a \perp (e - a)$ and the same with $b$ (see more preliminary information on Riesz spaces in Subsection~\ref{subsfragm}). It is well known that the set $\mathfrak{F}_e$ of all fragments of $e$ is a Boolean algebra with unity $e$. Thus, we deal with an arbitrary Boolean algebra $\mathcal{B}$ to define a Rademacher system, which becomes a sequence (or, more generally, a family) of two-element partitions of unity $\mathbf{1}$ of $\mathcal{B}$. For convenience of notation, instead of a sequence of partitions $\mathbf{1} = r_n \sqcup s_n$ we consider a sequence $(r_n)$ of representatives of each partition, no matter which ones. To distinguish Rademacher systems in Riesz spaces and Boolean algebras, the later ones we call Rademacher families.

We find conditions on a family $(r_i)_{i \in I}$ of elements of a Boolean algebra $\mathcal{B}$ under which one can consider it as a Rademacher family. Moreover, if a Rademacher family $(r_i)_{i \in I}$ in $\mathcal{B}$ has an additional property of completeness, there is a unique countably additive positive (that is, strictly positive at every nonzero element) measure on $\mathcal{B}$ possessing the equality
\begin{equation} \label{eq:main}
\mu \Bigl( \bigcap_{k=1}^n \theta_k r_{i_k} \Bigr) = \frac{1}{2^n}
\end{equation}
for all finite collections of distinct indices $i_1, \ldots, i_n \in I$ and all collections of signs $\theta_1, \ldots, \theta_n \in \{-1,1\}$, where the product $\theta x$ of a sign $\theta \in \{-1,1\}$ by an element $x \in \mathcal{B}$ is defined by setting $1 x = x$ and $-1 x = - x = \mathbf{1} \setminus x$.

Since a complete Rademacher family defines a positive measure on a Boolean algebra by \eqref{eq:main}, a Boolean algebra admitting a complete Rademacher family must be a measurable algebra. Moreover, the existence of a complete Rademacher family is a necessary and sufficient condition on a Boolean algebra to be measurable and Maharam homogeneous.

Such an approach to define a Rademacher system is then used to construct different classical systems in an arbitrary Riesz space, as well as to develop a theory of integration in a Riesz space $E$ of elements of the band $B_e$ generated by a fixed $0 < e \in E$ with respect to the measure on the Boolean algebra $\mathfrak{F}_e$ of fragments of $e$ generated by a complete Rademacher family on $\mathfrak{F}_e$.

$\,\,\,\,\,\,$\textbf{Organization of the paper.} Necessary information on Boolean algebras is put to the introduction, and preliminaries on Riesz spaces the reader can find in Section~\ref{secfive}.

In Section~\ref{sec:crs} we analyze properties of the usual Rademacher system as a family of two-point partitions of unity in a Boolean algebra to highlight the most important ones that may be put to the definition of an abstract Rademacher system. We outline the desired properties and show that none of them can be obtained from the other ones. Then we define a pre-Rademacher family as a system of elements which has one of the needed properties. More precisely, a system $\mathcal{R} = (r_i)_{i \in I}$ in a Boolean algebra $\mathcal{B}$ is pre-Rademacher if for any finite subset $J \subset I$ and any choice of signs $\theta_j = \pm 1$, $j \in J$ one has $\inf_{j \in J} \theta_j r_j > \mathbf{0}$. It is worthwhile to consider pre-Rademacher families because a number of results on Rademacher systems uses only this property in their proofs. The main results of the section (Theorem~\ref{th:eqcard}) characterizes atoms and atomless subalgebras in terms of pre-Rademacher systems. Theorem~\ref{thm:dyadicfam} asserts that a pre-Rademacher system $\mathcal{R} = (r_i)_{i \in I}$ in a Boolean algebra $\mathcal{B}$ uniquely determines a finitely additive measure $\mu$ satisfying \eqref{eq:main} on the minimal subalgebra $\mathcal{B}(\mathcal{R})$ containing this system. We call such a measure the dyadic measure on the subalgebra $\mathcal{B}(\mathcal{R})$ generated by $\mathcal{B}$. The last two results of Section~\ref{sec:crs} (Theorem~\ref{th:kjhgf5} and its generalization Theorem~\ref{th:kjhgf7}) show that maximal pre-Rademacher systems in all natural cases cannot be considered as independent random variables.

Section~\ref{sec:Radfam} aims to formalize the notion of a Rademacher system. The independence of random variables we interpret as minimality of a system, that is, any element of the system does not belong to the smallest order closed subalgebra generated by the rest of elements. The property of a system $\mathcal{R} = (r_i)_{i \in I}$ in a Boolean algebra $\mathcal{B}$ to be identically distributed we identify with the so-called vanishing at infinity property, that is, for any infinite subset $J \subseteq I$ and any choice of signs $\theta_j = \pm 1$, $j \in J$ one has $\inf_{j \in J} \theta_j r_j = \mathbf{0}$. Then by a Rademacher family we mean a minimal vanishing at infinity pre-Rademacher family. To emphasize the importance of the vanishing at infinity property of a Rademacher family, we provide examples showing the variety of distinct (non-isomorphic) types of countable minimal pre-Rademacher families without this property (see Proposition~\ref{ex:5eqcardex} and Theorem~\ref{thm:5eqcardex}). Another sense of the vanishing at infinity property is explained below in Theorem~\ref{th:ccvmn740}: this property gives positivity of the dyadic measure generated by a pre-Rademacher family. The main result of Section~\ref{sec:Radfam} (Theorem~\ref{th:ccvmn740}) asserts that, if $\mathcal{R} = (r_i)_{i \in I}$ is a minimal pre-Rademacher family in a Boolean algebra $\mathcal{B}$ then the dyadic finitely additive measure on $\mathcal{B}(\mathcal{R})$ is uniquely extended to a countable additive measure on the smallest order $\sigma$-closed subalgebra $\mathcal{B}_\sigma(\mathcal{R})$ which we also call the dyadic measure. If, moreover, $\mathcal{R}$ vanishes at infinity (in other words, is a Rademacher family) then the dyadic measure is positive. It ought be mentioned (see the remark between the statement and the proof of Theorem~\ref{th:ccvmn740}) that the dyadic measure is not countably additive on $\mathcal{B}(\mathcal{R})$ in the most natural cases (the restriction of a countably additive measure on a Boolean algebra to a subalgebra need not be countably additive; that is an essential difference in measure theory on Boolean algebras from the measure theory on measure spaces). We say that a Rademacher system is complete if it $\sigma$-generates the entire Boolean algebra. As a corollary of Theorem~\ref{th:ccvmn740} we obtain that a complete Rademacher system on a Boolean algebra $\mathcal{B}$ defines a unique dyadic countably additive probability measure on $\mathcal{B}$ which is positive (that is, $\mu(x) > 0$ for all $x \in \mathcal{B} \setminus \{\mathbf{0}\}$).

In Section~\ref{sec:charhma} we give a new characterization of homogeneous measurable Boolean algebras. First we prove that Dedekind $\sigma$-complete Boolean algebras admitting complete Rademacher systems of the same cardinality are isomorphic. Then, as a consequence of the above result and the Maharam theorem we obtain that a Dedekind $\sigma$-complete Boolean algebra is homogeneous measurable if and only if it contains a complete Rademacher family.

Section~\ref{secfive} is devoted to applications of the obtained results to Riesz spaces. We define a Rademacher system in a Riesz space $E$ (on an element $0 < e \in E$) in a natural and obvious way using a Rademacher family in the Boolean algebra $\mathfrak{F}_e$ of fragments of $e$. Theorem~\ref{th:ssssmdnd750} asserts that a complete Rademacher system of a cardinality $\aleph_\alpha$ on $e$ exists if and only if the ideal $A_e$ in $E$ generated by $e$ is Riesz isomorphic to $L_\infty(D^{\omega_\alpha}, \Sigma_{\omega_\alpha}, \mu_{\omega_\alpha})$. Then we define a Haar system with respect to a given countable Rademacher system in obvious way and prove that a complete Haar system on an element $e$ of a Dedekind complete Riesz space $E$ is an order Schauder basis of $A_e$. At the end of the section we pose the following problem. Let $E$ be a Riesz space, $0 < e \in E$ and let $(h_n)_{n=1}^\infty$ be a complete Haar system on $e$. Characterize those Riesz subspaces $X$ of $B_e$ such that $(h_n)_{n=1}^\infty$ is an order Schauder basis of $E$.

In Section~\ref{sec:int} we give another application to the theory of Riesz spaces. More precisely, we define an integral $\int_e x \, d \mu$ of an element $x$ over an element $e > 0$ of a Dedekind $\sigma$-complete Riesz space with respect to the dyadic measure $\mu$ generated by a complete Rademacher system $\mathcal{R}$ on $e$, and prove some usual properties of the integral. We also define Riesz subspaces $L_1(\mathcal{R})$ of $E$ in an obvious way which are ideals of $E$, need not be bands of $E$, are normed lattices with respect to the norm $\|x\| = \int_e |x| \, d \mu$ and need not be Banach lattices. For the most of the results of the section, the proofs are standard, and so we omit them or provide with a sketch.

In the last Section~\ref{prmainpartans}, we formulate a partial answer to Problem~\ref{pr:main} on the base of previous results.

$\,\,\,\,\,\,$\textbf{Acknowledgments.} We thank T.~Banakh, A.~Dorogovtsev and O.~Maslyuchenko for helpful discussions on the subject.

$\,\,\,\,\,\,$\textbf{Preliminaries.} \label{ss:dkdkj} Our terminology is standard, mainly as in \cite{F3} for Boolean algebras and as in \cite{ABu} for Riesz spaces. \emph{Zero} $\mathbf{0}$ and \emph{unit} $\mathbf{1}$ of a Boolean algebra $\mathcal{B}$ we write in bold to distinguish them from the corresponding numbers. The order $x \leq y$ on $\mathcal{B}$ is defined to be equivalent to the equality $x \cap y = x$, which in turn is equivalent to $x \cup y = y$. So, $x \cup y = \sup\{x,y\}$ and $x \cap y = \inf\{x,y\}$ with respect to this order. The relation $\subseteq$ is used for subsets and $\leq$ is used for elements of a Boolean algebra. The union and the intersection of an infinite subset $\mathcal{A} \subseteq \mathcal{B}$ is defined by $\bigcup \mathcal{A} = \sup \mathcal{A}$ and $\bigcap \mathcal{A} = \inf \mathcal{A}$ with respect to the order $\leq$ only if the corresponding supremum or infimum exists. By a \emph{subalgebra} of $\mathcal{B}$ we mean any subset of $\mathcal{B}$ containing $\mathbf{1}$ which is itself a Boolean algebra with the induced Boolean algebra structure. A subset $\mathcal{A}$ of a Boolean algebra $\mathcal{B}$ is said to be \emph{disjoint} provided $x \cap y = \mathbf{0}$ for all distinct $x,y \in \mathcal{A}$. By a \emph{partition} (of unity) in a Boolean algebra $\mathcal{B}$ we mean a maximal disjoint subset $\mathcal{A} \subseteq \mathcal{B}$, that is, $(\forall x \in \mathcal{B} ) \, \bigl(\bigl( \forall a \in \mathcal{A}  \,\, a \cap x = \mathbf{0} \bigr) \Rightarrow  (x = \mathbf{0}) \bigr).$ A \emph{disjoint union} $\bigcup \mathcal{A}$ (that is, the union of a disjoint system $\mathcal{A} \subseteq \mathcal{B}$), if exists, is denoted by $\bigsqcup \mathcal{A}$. Although in some cases an infinite union in a Boolean algebra does not exist, it is immediate that if $\mathcal{A}$ is a partition then $\bigsqcup \mathcal{A} = \mathbf{1}$ exists. Conversely, if $\bigsqcup \mathcal{A} = \mathbf{1}$ then $\mathcal{A}$ is a partition. A Boolean algebra $\mathcal{B}$ is said to have the \emph{countable chain condition} (ccc, in short) if any disjoint subset $\mathcal{A} \subseteq \mathcal{B}$ is, at most, countable. A Boolean algebra $\mathcal{B}$ is called \emph{measurable} if $\mathcal{B}$ is a Dedekind $\sigma$-complete Boolean algebra and there is a finite positive $\sigma$-additive measure on $\mathcal{B}$ (by a \emph{positive} measure we mean a strictly positive measure $\mu$, that is, $\mu(x) > 0$ for all $x \in \mathcal{B} \setminus \{\mathbf{0}\}$). Obviously, every measurable Boolean algebra has the ccc.

Another well known necessary condition on a Boolean algebra $\mathcal{B}$ to be measurable is the weak distributivity property, that is, for any sequence $\mathcal{A}_n$ of partitions there is a partition $\mathcal{A}$ such that every element of $\mathcal{A}$ for every $n$ meets only a finite number of elements of $\mathcal{A}_n$ (we say that $a$ meets $b$ provided $a \cap b \neq \mathbf{0}$). In \cite{Tal} Talagrand constructed a Boolean algebra with the CCC and the weak distributivity property failing to be measurable, answering Problem~163 by von Neumann from the Scottish Book. Very soon Jech \cite{Jech} gave nice characterizations of measurable Boolean algebras.

A Boolean algebra $\mathcal{B}$ is said to be Dedekind complete (respectively, $\sigma$-Dedekind complete) if every nonempty (respectively, every nonempty countable) subset of $\mathcal{B}$ has the supremum (equivalently, infimum).

For example, the Borel $\sigma$-algebra $\mathcal{B}$ on $[0,1]$ is not Dedekind complete. Moreover, $\mathcal{B}$ does not have the ccc and hence, is not measurable. However, the $\sigma$-algebra $\mathcal{\widehat{B}}$ of all equivalence classes of Borel subsets of $[0,1]$ with respect to the Lebesgue measure is measurable (the Lebesgue measure is positive on $\mathcal{\widehat{B}}$) and Dedekind complete. \textbf{Throughout the paper, we reserve the notation $\mathcal{\widehat{B}}$ for the above Boolean algebra, which will be frequently used in different contexts.} In fact, every Boolean Dedekind $\sigma$-complete Boolean algebra with the ccc is Dedekind complete.

A subalgebra $\mathcal{A}$ of a Boolean algebra $\mathcal{B}$ is said to be \emph{order closed} (respectively, $\sigma$-\emph{order closed}) if for every (respectively, countable) subset $C \subseteq \mathcal{A}$ the existence of $\sup C \in \mathcal{B}$ implies $\sup C \in \mathcal{A}$. Another equivalent definition contains an additional assumption on $C$ to be upwards directed (for this and other equivalences see \cite[313E]{F3}).

For any $\mathcal{A} \subseteq \mathcal{B}$ we denote
\begin{itemize}
  \item $\mathcal{B} (\mathcal{A})$ the smallest subalgebra of $\mathcal{B}$ including $\mathcal{A}$;
  \item $\mathcal{B}_\sigma (\mathcal{A})$ the smallest $\sigma$-order closed subalgebra of $\mathcal{B}$ including $\mathcal{A}$;
  \item $\mathcal{B}_\tau (\mathcal{A})$ the smallest order closed subalgebra of $\mathcal{B}$ including $\mathcal{A}$.
\end{itemize}

Anyway, $\mathcal{B} (\mathcal{A}) \subseteq \mathcal{B}_\sigma (\mathcal{A}) \subseteq \mathcal{B}_\tau (\mathcal{A})$, and if $\mathcal{B}$ possesses the ccc then $\mathcal{B}_\sigma (\mathcal{A}) = \mathcal{B}_\tau (\mathcal{A})$ for all subsets $\mathcal{A} \subseteq \mathcal{B}$ \cite[331G]{F3}.

Standard arguments show that, given a subalgebra $\mathcal{A}$ of a Boolean algebra $\mathcal{B}$, the subalgebra $\mathcal{B}_\tau (\mathcal{A})$ equals the order closure of $\mathcal{A}$ in $\mathcal{B}$ \cite[313F(c)]{F3}. Hence, $\mathcal{B}_\tau (\mathcal{A})$ equals the set of all order limits of nets from $\mathcal{A}$. Recall that an element $x$ of a Boolean algebra $\mathcal{B}$ is called the order limit of a net $(x_\alpha)$ in $\mathcal{B}$ if there exists a net $(u_\alpha)$ with the same index set in $\mathcal{B}$ such that $x_\alpha \triangle x \leq u_\alpha$ for all $\alpha$ and $u_\alpha \downarrow \mathbf{0}$ (the latter condition means that $(u_\alpha)$ is downwards directed and $\inf_\alpha u_\alpha = \mathbf{0}$). If $x$ is the order limit of a net $(x_\alpha)$ then we say that $(x_\alpha)$ order converges (to $x$) and write $x_\alpha \stackrel{\rm o}{\longrightarrow} x$. Analogously, $\mathcal{B}_\sigma (\mathcal{A})$ equals the set of all order limits of sequences from $\mathcal{A}$. The order convergence has the usual properties, including the order continuity of the operations $\cup$, $\cap$, $\triangle$, $\setminus$.

Given any $\mathcal{A} \subseteq \mathcal{C} \subseteq \mathcal{B}$, we say that
\begin{itemize}
  \item $\mathcal{A}$ $\sigma$-generates $\mathcal{C}$ if $\mathcal{C} = \mathcal{B}_\sigma (\mathcal{A})$;
  \item $\mathcal{A}$ $\tau$-generates $\mathcal{C}$ if $\mathcal{C} = \mathcal{B}_\tau (\mathcal{A})$.
\end{itemize}

We define the \emph{density} ${\rm dens} \, \mathcal{B}$ of a Boolean algebra $\mathcal{B}$ to be the smallest cardinality of subsets $\mathcal{A} \subseteq \mathcal{B}$ that $\tau$-generate $\mathcal{B}$. The \emph{density} ${\rm dens} \,e$ of a nonzero element $e \in \mathcal{B}$ is defined to be the density of the Boolean algebra $\mathcal{B}_e = \{x \in \mathcal{B}: \, x \leq e\}$ with operations induces by $\mathcal{B}$ and unit $e$. In particular, ${\rm dens} \, \mathbf{1} = {\rm dens} \, \mathcal{B}$. We say that a Boolean algebra $\mathcal{B}$ is \emph{homogeneous}\footnote{more precisely, \emph{Maharam homogeneous}} if for every $e \in \mathcal{B} \setminus \{0\}$ we have ${\rm dens} \, e = {\rm dens} \, \mathcal{B}$.

Set $D = \{-1,1\}$ and let $\omega_\alpha$ be an arbitrary infinite cardinal, $\mu_{\omega_\alpha}$ the Haar measure on the $\sigma$-algebra $\Sigma_{\omega_\alpha}$ of subsets of the power-set $D^{\omega_\alpha}$ considered as a compact abelian group, $\widehat{\Sigma}_{\omega_\alpha}$ the quotient Boolean algebra of $\Sigma_{\omega_\alpha}$ modulo $\mu_{\omega_\alpha}$-null sets. The quotient map from $\Sigma_{\omega_\alpha}$ to $\widehat{\Sigma}_{\omega_\alpha}$ we denote by ${\rm Co}$. By the \emph{standard Rademacher family} $(\overline{r}_\gamma)_{\gamma < \omega_\alpha}$ in $\widehat{\Sigma}_{\omega_\alpha}$ we mean the co-sets of the following sets: $\overline{r}_\gamma = {\rm Co} \, \bigl\{ x \in D^{\omega_\alpha}: \,\, x(\gamma) = 1 \bigr\}.$

The following results due to Maharam is very essential for our investigation (see \cite{Maharam} for the original paper, and \cite{F3}, \cite{Lacey}, \cite{PP} for different proofs).

\begin{thm}[First Maharam Theorem] \label{th:Mah1}
Any homogeneous measurable Boolean algebra $\mathcal{B}$ is isomorphic to $\widehat{\Sigma}_{\omega_\alpha}$ where $\omega_\alpha$ is the cardinal of cardinality $\aleph_\alpha = {\rm dens} \, \mathcal{B}$.
\end{thm}

The importance of homogeneous measurable algebras is explained in the next theorem.

\begin{thm}[Second Maharam Theorem] \label{th:Mah2}
Any atomless measurable Boolean algebra $\mathcal{B}$ is isomorphic to an at most countable direct sum $\bigoplus_{i \in I} \mathcal{B}_i$ of homogeneous measurable algebras $\mathcal{B}_i$, which, in turn, is isomorphic to $\bigoplus_{i \in I} \widehat{\Sigma}_{\omega_{\alpha_i}}$, where $\aleph_{\alpha_i} = {\rm dens} \, \mathcal{B}_i$ for all $i \in I$ by the first Maharam Theorem.
\end{thm}

By a \emph{semialgebra} in a Boolean algebra $\mathcal{B}$ we mean a subset $P \subseteq \mathcal{B}$ possessing the following properties:
\begin{enumerate}
  \item $\mathbf{0}, \mathbf{1} \in P$;
  \item if $a,b \in P$ then $a \cap b \in P$;
  \item if $a_1,b \in P$ with $a_1 \subseteq b$ then there are $n \geq 1$ and $a_2, \ldots, a_n \in \mathcal{P}$ (if $n > 1$) such that $b = \bigsqcup_{m = 1}^n a_m$.
\end{enumerate}

We will use the following description of the order closed subalgebra generated by a semialgebra.

\begin{prop} \label{pr:semigen}
Let $\mathcal{A}$ be a semialgebra in a Boolean algebra $\mathcal{B}$. Then
\begin{enumerate}
  \item $\mathcal{B}(\mathcal{A})$ equals the set of all finite disjoint unions of elements of $\mathcal{B}$;
  \item $\mathcal{B}_\tau(\mathcal{A})$ equals the set of all order limits of elements of $\mathcal{B}(\mathcal{A})$.
\end{enumerate}
\end{prop}

To prove item (1) of Proposition~\ref{pr:semigen} is a standard technical exercise (see \cite[Lemma~1.2.14]{Bog}). Item (2) follows from (1) and the fact that the order closure of a subalgebra is a subalgebra \cite[313F(c)]{F3}.

\section{Pre-Rademacher families} \label{sec:crs}

$\,\,\,\,\,\,$\textbf{The usual Rademacher family on a dyadic interval.} Let $I_n^k = \bigl[\frac{k-1}{2^n}, \frac{k}{2^n} \bigr)$ be the dyadic intervals, $n = 0,1,2, \ldots$, $k = 1, \ldots, 2^n$. We set
\begin{equation} \label{eqzxxxm}
r_n = \bigsqcup_{j=1}^{2^{n-1}} I_n^{2j-1}, \,\,\,\,\, n = 1,2, \ldots.
\end{equation}
By the \emph{usual Rademacher family} on $[0,1)$ we mean the sequence $(\widehat{r}_n)_{n \in \mathbb N}$ of the co-sets of $r_n$ in the quotient Boolean algebra $\widehat{\mathcal{B}}$ modulo measure null sets of the Borel $\sigma$-algebra $\mathcal{B}$ on $[0,1)$. Observe that the usual Rademacher family $\widehat{\mathcal{R}} = (\widehat{r}_n)_{n \in \mathbb N}$ possesses the following properties:
\begin{enumerate}
  \item[(R1)] $\bigcap_{j \in J} \theta_j \widehat{r}_j \neq \mathbf{0}$ for any finite subset $J \subset \mathbb N$ and any collection of signs $\theta_j = \pm 1$, $j \in J$;
  \item[(R2)] $\bigcap_{j \in J} \theta_j \widehat{r}_j = \mathbf{0}$ for any infinite subset $J \subseteq \mathbb N$ and any collection of signs $\theta_j = \pm 1$, $j \in J$;
  \item[(R3)] for any $n_0 \in \mathbb N$ one has $\widehat{\mathcal{B}}_\tau \bigl((\widehat{r}_n)_{n \in \mathbb N \setminus \{n_0\}}\bigr) \neq \widehat{\mathcal{B}}_\tau (\widehat{\mathcal{R}})$;
  \item[(R4)] $\widehat{\mathcal{B}}_\tau (\widehat{\mathcal{R}}) = \widehat{\mathcal{B}}$.
\end{enumerate}

We are going to introduce a complete Rademacher family in a Boolean algebra as a family of two-point partitions of unity to satisfy properties (R1)-(R4). More precisely,
\begin{itemize}
  \item (R1) determines a pre-Rademacher family;
  \item (R1)\&(R2) determine a vanishing at infinity pre-Rademacher family;
  \item (R1)\&(R3) determine a minimal pre-Rademacher family;
  \item (R1)\&(R4) determine a complete pre-Rademacher family;
  \item (R1)\&(R2)\&(R3) determine a Rademacher family;
  \item (R1)\&(R2)\&(R3)\&(R4) determine a complete Rademacher family.
\end{itemize}

We define the \emph{usual Rademacher family on a dyadic interval} $I_m^j$, $m \in \mathbb N$, $j \in \{1, \ldots, 2^m\}$, to be the subsequence $(\widehat{r}_n{'})_{n \in \mathbb N}$ of $(\widehat{r}_n)_{n \in \mathbb N}$ consisting of all those $\widehat{r}_n$ for which $\widehat{r}_n \leq \widehat{r}_{2^m+j}$ and numbered by all positive integers. So, formally the usual Rademacher family on any dyadic interval consists of elements of $\widehat{\mathcal{B}}$. Note that the usual Rademacher family on any dyadic interval has properties (R1), (R2) and (R3) and does not have (R4) if the dyadic interval is not $[0,1)$.

We provide below with examples showing that none of properties (R1)-(R4) follows from the rest ones even for the Boolean algebra $\widehat{\mathcal{B}}$ (of course, when saying of a sequence $\mathcal{R} = (r_n)_{n=1}^\infty$ we mean properties (R1)-(R4) adapted to that sequence).

\begin{ex} \label{ex:ffk1}
There is a sequence $\mathcal{R} = (r_n)_{n=1}^\infty$ in $\widehat{\mathcal{B}}$ satisfying (R2), (R3) and (R4) and failing (R1).
\end{ex}

\begin{proof}
Let $(\widehat{r}_n{'})_{n \in \mathbb N}$ and $(\widehat{r}_n{''})_{n \in \mathbb N}$ be the usual Rademacher families on $\bigl[0,\frac12\bigr)$ and $\bigl[\frac12,1\bigr)$ respectively. We define a sequence $(r_n)_{n \in \mathbb N}$ in $\widehat{\mathcal{B}}$ by setting $r_{2k-1} = \widehat{r}_k{'}$ and $r_{2k} = \widehat{r}_k{''}$ for $k = 1,2, \ldots$. Then $\mathcal{R} = (r_n)_{n=1}^\infty$ has the desired properties.
\end{proof}

\begin{ex} \label{ex:ffk2}
There is a sequence $\mathcal{R} = (r_n)_{n=1}^\infty$ in $\widehat{\mathcal{B}}$ satisfying (R1), (R3) and (R4) and failing (R2).
\end{ex}

\begin{proof}
Let $(\widehat{r}_n{'})_{n \in \mathbb N}$ and $(\widehat{r}_n{''})_{n \in \mathbb N}$ be the usual Rademacher families on $\bigl[0,\frac12\bigr)$ and $\bigl[\frac12,1\bigr)$ respectively. We define a sequence $(r_n)_{n \in \mathbb N}$ in $\widehat{\mathcal{B}}$ by setting $r_{2k-1} = \widehat{[0,\frac12\bigr)} \sqcup \widehat{r}_{2k-1}{''}$ and $r_{2k} = \widehat{r}_k{'} \sqcup \widehat{r}_{2k}{''}$ for $k=1,2,\ldots$. Then $\mathcal{R} = (r_n)_{n \in \mathbb N}$ satisfies (R1). Indeed, for any $J \subseteq \mathbb N$ and any collection of signs $\theta_j = \pm 1$, $j \in J$, one has
$$
\bigcap_{j \in J} \theta_j r_j \geq \bigcap_{j \in J} \left( \widehat{\Bigl[ \frac12,1 \Bigr)} \cap \theta_j r_j \right) = \widehat{\Bigl[ \frac12,1 \Bigr)} \cap \bigcap_{j \in J} \theta_j \widehat{r}_j{''} \neq \mathbf{0}.
$$

$\mathcal{R}$ does not satisfy (R2) because $\bigcap_{k=1}^\infty r_{2k-1} = \bigl[1,\frac12\bigr) \neq \mathbf{0}$.

$\mathcal{R}$ satisfies (R3) because $(\widehat{r}_n{''})_{n \in \mathbb N}$ satisfies (R3).

(R4) for $\mathcal{R}$ follows from the observation that $\widehat{\mathcal{B}}_\tau (\mathcal{R})$ contains every dyadic interval.
\end{proof}

\begin{ex} \label{ex:ffk3}
There is a sequence $\mathcal{R} = (r_n)_{n=1}^\infty$ in $\widehat{\mathcal{B}}$ satisfying (R1), (R2) and (R4) and failing (R3).
\end{ex}

The existence of a family satisfying the claims of Example~\ref{ex:ffk3} is not so obvious and follows from Theorem~\ref{th:kjhgf5} below.

\begin{ex} \label{ex:ffk4}
There is a sequence $\mathcal{R} = (r_n)_{n=1}^\infty$ in $\widehat{\mathcal{B}}$ satisfying (R1), (R2) and (R3) and failing (R4).
\end{ex}

\begin{proof}
Just remove from the usual Rademacher family any element.
\end{proof}

Given a subset $\mathcal{R}$ of a Boolean algebra $\mathcal{B}$, a \emph{particle} (with respect to $\mathcal{R}$) is any element $\bigcap_{j \in J} \theta_j r_j$, where $J$ is a finite index set, $\theta_j \in \{-1,1\}$ and $r_j \in R$ are distinct elements, $j \in J$. By convention, the particle constructed by the empty set $J = \emptyset$, as well as the infimum over the empty set, is considered to be unit $\mathbf{1}$ of $\mathcal{B}$.

Observe that if $\{r_1, \ldots, r_n\}$ is a finite subset of $\mathcal{R}$ then
\begin{equation} \label{eq:beg}
\mathbf{1} = \bigsqcup_{\theta_j = \pm 1} \bigcap_{j=1}^n \theta_j r_j.
\end{equation}

\begin{defn}
A subset $\mathcal{R}$ of a Boolean algebra $\mathcal{B}$ is called a \emph{pre-Rademacher family} if all particles with respect to $\mathcal{R}$ are nonzero.
\end{defn}

Evidently, every subfamily of a pre-Rademacher family is a pre-Rademacher family.

$\,\,\,\,\,\,$\textbf{Particles semialgebra.} \begin{prop} \label{pr:jfggdd5567}
Let $\mathcal{R}$ be a pre-Rademacher family in a Boolean algebra $\mathcal{B}$ and $\mathcal{P}$ the set of all particles with respect to $\mathcal{R}$. Then $\overline{\mathcal{P}} = \mathcal{P} \cup \{\mathbf{0}\}$ is a semialgebra.
\end{prop}

\begin{proof}
Obviously, $\mathbf{0}, \mathbf{1} \in \overline{\mathcal{P}}$ and $\overline{\mathcal{P}}$ is closed under intersections. We prove (3). Let $a_1, b \in \mathcal{P}$ with $a_1 \subseteq b$. Assume, for convenience, $\mathcal{R} = (r_i)_{i \in I}$ with $r_i \neq r_j$ for $i \neq j$. Let $J,K$ be any finite subsets of $I$ and
\begin{equation} \label{eq:beg2}
\mathbf{0} \neq a_1 = \bigcap_{j \in J} \theta_j' r_j \subseteq \bigcap_{k \in K} \theta_k'' r_k = b.
\end{equation}

First we show that $K \subseteq J$. Indeed, if $\exists k \in K \setminus J$ then $a_1 = (a_1 \cap \theta_k'' r_k) \cup (a_1 \cap - \theta_k'' r_k)$. Since all particles are nonzero, $\mathbf{0} \neq a_1 \cap - \theta_k'' r_k \subseteq a_1 \subseteq b \subseteq \theta_k'' r_k, $ which contradicts $- \theta_k'' r_k \cap \theta_k'' r_k = 0$.

Now we prove that $(\forall k \in K) \, \theta_k'' = \theta_k'$. Indeed, if $(\exists k \in K) \, \theta_k'' \neq \theta_k'$ then intersecting both sides of \eqref{eq:beg2} with $\theta_k' r_k$ we would obtain $\mathbf{0} \neq a_1 \subseteq b \cap \theta_k'' r_k \cap \theta_k' r_k = \mathbf{0},$ a contradiction. By \eqref{eq:beg}, $\mathbf{1} = \bigsqcup_{\theta_j = \pm 1} \bigcap_{i \in J \setminus K} \theta_i r_i$ and hence
\begin{align*}
b &= \bigsqcup_{\theta_j = \pm 1} \Bigl(b \cap \bigcap_{i \in J \setminus K} \theta_i r_i\Bigr) = \bigsqcup_{\theta_j = \pm 1} \Bigl(\bigcap_{k \in K} \theta_k'' r_k \cap \bigcap_{i \in J \setminus K} \theta_i r_i\Bigr) = \bigsqcup_{\theta_j = \pm 1} \Bigl(\bigcap_{k \in K} \theta_k' r_k \cap \bigcap_{i \in J \setminus K} \theta_i r_i\Bigr), \\
\end{align*}
and one of the summands equals $a_1$.
\end{proof}

\begin{defn}
Let $\mathcal{R}$ be a pre-Rademacher family in a Boolean algebra $\mathcal{B}$ with the set $\mathcal{P}$ of all particles. The semialgebra $\overline{\mathcal{P}} = \mathcal{P} \cup \{\mathbf{0}\}$ is called the \emph{particle semialgebra} of the family $\mathcal{R}$.
\end{defn}

$\,\,\,\,\,\,$\textbf{Atoms of the order closed subalgebra generated by a pre-Rademacher family.}
Recall that a nonzero element $a$ of a Boolean algebra $\mathcal{B}$ is called an atom if for any $x \in \mathcal{B}$ the inclusion $x \subseteq a$ implies that either $x = \mathbf{0}$ or $x = a$.

\begin{defn} \label{def015}
A pre-Rademacher family $\mathcal{R}$ in a Boolean algebra $\mathcal{B}$ is called:
\begin{itemize}
  \item $\sigma$-\emph{complete} if $\mathcal{R}$ $\sigma$-generates $\mathcal{B}$;
  \item $\tau$-\emph{complete} if $\mathcal{R}$ $\tau$-generates $\mathcal{B}$;
  \item $\sigma$-\emph{atomless} if $\mathcal{B_\sigma(\mathcal{R})}$ is atomless;
  \item $\tau$-\emph{atomless} if $\mathcal{B_\tau(\mathcal{R})}$ is atomless.
\end{itemize}
If, in addition, $\mathcal{R}$ has the ccc, we say \emph{atomless} or \emph{complete} for both $\sigma$- and $\tau$- versions\footnote{remind that for ccc algebras one has $\mathcal{B}_\sigma (\mathcal{R}) = \mathcal{B}_\tau (\mathcal{R})$}.
\end{defn}

The proofs of following observations are straightforward.

\begin{remark} \label{rem:ddk}
$\,$
\begin{enumerate}
  \item The standard Rademacher family $(\overline{r}_\gamma)_{\gamma < \omega_\alpha}$ in $\widehat{\Sigma}_{\omega_\alpha}$ is an atomless complete pre-Rademacher family.
  \item The above properties of a pre-Rademacher family are preserved under an isomorphism.
\end{enumerate}
\end{remark}

As we will see later, a subsequence of an atomless countable pre-Rademacher family need not be atomless (see item (1) of Remark~\ref{rem:xbbd}).

\begin{thm} \label{th:eqcard}
Let $\mathcal{R} = (r_i)_{i \in I}$ be an infinite pre-Rademacher family in a Boolean algebra $\mathcal{B}$. Then
\begin{enumerate}
  \item for every nonzero element $a$ of $\mathcal{B_\tau(\mathcal{R})}$ the following assertions are equivalent
\begin{enumerate}
  \item $a$ is an atom in $\mathcal{B_\tau(\mathcal{R})}$;
  \item there is a collection of signs $(\overline{\theta}_i)_{i \in I}$ such that $a = \bigcap_{i \in I} \overline{\theta}_i r_i$.
\end{enumerate}
  \item the following assertions are equivalent
  \begin{enumerate}
    \item $\mathcal{B_\tau(\mathcal{R})}$ is atomless;
    \item for every collection of signs $\theta_i = \pm 1$ one has that either $\bigcap_{i \in I} \theta_i r_i = \mathbf{0}$ or $\bigcap_{i \in I} \theta_i r_i$ does not exist.
  \end{enumerate}
\end{enumerate}
\end{thm}

For the proof, we need two lemmas. Let $\overline{\mathcal{P}}$ be the semialgebra of $\mathcal{R}$, $\Theta = (\overline{\theta}_i)_{i \in I}$ a fixed collection of signs. We set
$$
\mathcal{P}_\Theta = \Bigl\{ \bigcap_{j \in J} \overline{\theta}_j r_j: \,\, J \,\, \mbox{is a finite subset of} \,\, I \Bigr\}.
$$

\begin{lemma} \label{le:048uy}
Let $0 < z \in B_\tau(\mathcal{R})$ be any lower bound for $\{\overline{\theta}_i r_i: \, i \in I\}$. Then there is a net $(z_\alpha')$ in $\mathcal{P}_\Theta$ such that $z_\alpha' \stackrel{\rm o}{\longrightarrow} z$.
\end{lemma}

\begin{proof}
Since $z \in B_\tau(\mathcal{R})$, by Proposition~\ref{pr:semigen}, there exists a net $(z_\alpha)$ in $B_\tau(\mathcal{R})$ of the form $z_\alpha = \bigsqcup_{k=1}^{m_\alpha} p_{\alpha,k}$, where $p_{\alpha,k} = \bigcap_{i \in I_{\alpha,k}} \epsilon_{\alpha,k,i} r_i$ and $I_{\alpha,k}$ are finite subsets of $I$ and $\epsilon_{\alpha,k,i} = \pm1$, such that $z_\alpha \stackrel{\rm o}{\longrightarrow} z$. Given any $\alpha$, we set
$$
K_\alpha' = \{k \leq m_\alpha: \,\, (\forall i \in I_{\alpha,k}) \,\,\,\, \epsilon_{\alpha,k,i} = \overline{\theta}_i\}; \,\,\,\,\, K_\alpha'' = \{1, \ldots, m_\alpha\} \setminus K_\alpha';
$$
\begin{equation} \label{eq:dghfr0}
z_\alpha' = \bigsqcup_{k \in K_\alpha'} p_{\alpha,k} \,\,\,\,\, \mbox{and} \,\,\,\,\, z_\alpha'' = \bigsqcup_{k \in K_\alpha''} p_{\alpha,k}.
\end{equation}

Observe that $z_\alpha = z_\alpha' \sqcup z_\alpha''$. We show that $\forall \alpha$, $z_\alpha'' \cap z = \mathbf{0}$. Indeed, given any $\alpha$ and $k \in K_\alpha''$, we choose $i_0 \in I_{\alpha,k}$ so that $\epsilon_{\alpha,k,i_0} = - \overline{\theta}_{i_0}$ (the existence of such an index follows from the definition of $K_\alpha''$). Then
$$
p_{\alpha,k} \cap z = z \cap \bigcap_{i \in I_{\alpha,k}} \epsilon_{\alpha,k,i} r_i \leq \overline{\theta}_{i_0} r_{i_0} \cap \epsilon_{\alpha,k,i_0} r_{i_0} = \mathbf{0}.
$$

Hence,
$$
z_\alpha'' \cap z = \bigsqcup_{k \in K_\alpha''} p_{\alpha,k} \cap z = \mathbf{0}.
$$

Therefore,
\begin{equation} \label{eq:dghfr1}
z_\alpha' \cap z = (z_\alpha' \cap z) \cup (z_\alpha'' \cap z) = z_\alpha \cap z \stackrel{\rm o}{\longrightarrow} z \cap z = z.
\end{equation}

On the other hand,
\begin{equation} \label{eq:dghfr2}
z_\alpha' \cap -z \leq z_\alpha \cap -z \stackrel{\rm o}{\longrightarrow} z \cap -z = \mathbf{0}.
\end{equation}

By \eqref{eq:dghfr1} and \eqref{eq:dghfr2}, $z_\alpha' = (z_\alpha' \cap z) \sqcup (z_\alpha' \cap -z) \stackrel{\rm o}{\longrightarrow} z \cup 0 = z$. It remains to observe that, for every $\alpha$, the set $K_\alpha'$ is one-point. Indeed, let $j,k \in K_\alpha'$ with $j \neq k$. Then
$$
p_{\alpha,j} \cap p_{\alpha,k} = \bigcap_{i \in I_{\alpha,j}} \epsilon_{\alpha,j,i} r_i \cap \bigcap_{i \in I_{\alpha,k}} \epsilon_{\alpha,k,i} r_i = \bigcap_{i \in I_{\alpha,j}} \overline{\theta}_i r_i \cap \bigcap_{i \in I_{\alpha,k}} \overline{\theta}_i r_i = \bigcap_{i \in I_{\alpha,j} \cup I_{\alpha,k}} \overline{\theta}_i r_i \neq \mathbf{0}
$$
as a particle, which contradicts the disjointness of the first sum in \eqref{eq:dghfr0}.
\end{proof}

\begin{lemma} \label{le:048uy2}
Let $0 < a,b \in B_\tau(\mathcal{R})$ be lower bounds for $\{\overline{\theta}_i r_i: \, i \in I\}$. Then $a = b$.
\end{lemma}

\begin{proof}
Using Lemma~\ref{le:048uy}, choose nets $(x_\alpha)$ and $(y_\beta)$ from $\mathcal{P}_\Theta$ so that $x_\alpha \stackrel{\rm o}{\longrightarrow} a$ and $y_\beta \stackrel{\rm o}{\longrightarrow} b$. Let $(u_\alpha)$ and $(v_\beta)$ be nets in $\mathcal{B}$ so that $x_\alpha \triangle a \leq u_\alpha \downarrow \mathbf{0}$ and $y_\beta \triangle b \leq v_\beta \downarrow \mathbf{0}$. Since $a$ and $b$ are lower bounds for $\{\overline{\theta}_i r_i: \, i \in I\}$, we have that $p \cap a = a$ for all $p \in \mathcal{P}_\Theta$. In particular, $y_\beta \cap a = a$ for all $\beta$. Hence, for all $\alpha$ and $\beta$
\begin{equation} \label{eq:dghfr3}
(x_\alpha \cap y_\beta) \triangle a = (x_\alpha \cap y_\beta) \triangle (a \cap y_\beta) = (x_\alpha \triangle a) \cap y_\beta \leq x_\alpha \triangle a \leq u_\alpha.
\end{equation}

Analogously, for all $\alpha$ and $\beta$
\begin{equation} \label{eq:dghfr4}
(x_\alpha \cap y_\beta) \triangle b = (x_\alpha \cap y_\beta) \triangle (x_\alpha \cap b) = x_\alpha \cap (y_\alpha \triangle b) \leq y_\beta \triangle b \leq v_\beta.
\end{equation}

By \eqref{eq:dghfr3} and \eqref{eq:dghfr4}, for all $\alpha$ and $\beta$
$$
a \triangle b \leq \Bigl( a \triangle (x_\alpha \cap y_\beta) \Bigr) \cup \Bigl( (x_\alpha \cap y_\beta) \triangle b \Bigr) \leq u_\alpha \cup v_\beta.
$$

For any $\alpha$ by \cite[313B(b)]{F3}, $\inf_\beta u_\alpha \cup v_\beta = u_\alpha \cup \inf_\beta v_\beta = u_\alpha$. Hence, $a \triangle b \leq u_\alpha$ for all $\alpha$, which implies $a \triangle b = 0$.
\end{proof}

\begin{proof}[Proof of Theorem~\ref{th:eqcard}]
Observe that (2) is a direct consequence of (1). So, we prove (1). Let $0 < a \in \mathcal{B_\tau(\overline{\mathcal{P}})}$.

(a) $\Rightarrow$ (b). Fix any $i \in I$. Since $(r_i, -r_i)$ is a partition, either $a \cap r_i \neq \mathbf{0}$ or $a \cap - r_i \neq \mathbf{0}$. Therefore, since $a$ is an atom, either $a \leq r_i$ or $a \leq -r_i$. Set $\overline{\theta}_i = 1$ if $a \leq r_i$ and $\overline{\theta}_i = -1$ if $a \leq -r_i$. Thus, signs $(\overline{\theta}_i)_{i \in I}$ are chosen so that $(\forall i \in I) \,\,\,\, a \leq \overline{\theta}_i r_i$, that is, $a$ is a lower bound for $\{\overline{\theta}_i r_i: \, i \in I\}$. Show that $a = \bigcap_{i \in I} \overline{\theta}_i r_i$ (in particular, we show that the intersection exists), that is, $a$ is the greatest upper bound for $\{\overline{\theta}_i r_i: \, i \in I\}$, $a = \inf_{i \in I} \overline{\theta}_i r_i$.

Assume $x \in B_\tau(\mathcal{R})$ is any lower bound for $\{\overline{\theta}_i r_i: \, i \in I\}$. Our goal is to prove that $x \leq a$. Suppose on the contrary that $b = x \setminus a \neq \mathbf{0}$. Since $b \leq x$, we obtain that $b$ is a lower bound for $\{\overline{\theta}_i r_i: \, i \in I\}$. By Lemma~\ref{le:048uy2}, $a = b$ which contradicts the equality $b = x \setminus a$.

(b) $\Rightarrow$ (a). Assume $a = \bigcap_{i \in I} \overline{\theta}_i r_i$ and prove (a). Let $0 < b \in \mathcal{B_\tau(\mathcal{R})}$ be such that $b \leq a$. Then both $a, b$ are lower bounds for $\{\overline{\theta}_i r_i: \, i \in I\}$. By Lemma~\ref{le:048uy2}, $a = b$ and hence, $a$ is an atom in $\mathcal{B_\tau(\mathcal{R})}$.
\end{proof}

\begin{remark} \label{rem:xbbd}
Example~\ref{ex:ffk2} shows that:
\begin{enumerate}
  \item a subsequence of an atomless countable pre-Rademacher family need not be atomless;
  \item one cannot equivalently extend the claim of item (2)(b) in Theorem~\ref{th:eqcard} to any infinite intersection $\bigcap_{j \in J} \theta_j r_j = \mathbf{0}$ as far as in (R2).
\end{enumerate}
\end{remark}

The following example shows that the last possibility in item (b) of (2) in Theorem~\ref{th:eqcard} that $\bigcap_{i \in I} \theta_i r_i$ does not exist, can sometimes happen.

\begin{ex} \label{ex:ssjhd}
There exist a Boolean algebra $\mathcal{B}_0$ and a pre-Rademacher family $(s_n)_{n=1}^\infty$ in $\mathcal{B}$ with the following properties:
\begin{enumerate}
  \item[(i)] every subsequence of $(s_n)_{n=1}^\infty$ is an atomless pre-Rademacher family;
  \item[(ii)] the intersection $\bigcap_{n \in M} s_n$ does not exist for every infinite subset $M \subseteq \mathbb N$.
\end{enumerate}
\end{ex}

\begin{proof}
Let $\mathcal{B}_0$ be the subalgebra of $\widehat{\mathcal{B}}$ generated by the usual Rademacher family $(\widehat{r}_n)_{n=1}^\infty$. Fix any irrational number $\alpha \in (0,1)$ and choose a sequence $(D_n)_{n=1}^\infty$ of intervals $[a,b) \subseteq [\alpha,1)$, $a < b$ such that
\begin{enumerate}
  \item $\displaystyle{[\alpha,1) = \bigsqcup_{k=0}^{2^n-1} D_{2^n+k}}$ for all $n = 0,1, \ldots$;
  \item $\displaystyle{D_{2^n+k} = D_{2^{n+1}+2k-1} \sqcup D_{2^{n+1}+2k}}$ for all $n = 0,1, \ldots$ and $k = 0, \ldots, 2^n-1$;
  \item the endpoints of $D_n$ are dyadic numbers and $\alpha$ ordered in such a way that $\alpha = \min D_{2^n}$, $\sup D_{2^n+k} = \min D_{2^n+k+1}$ and $\sup D_{2^{n+1}-1} = 1$ for every $n = 0, 1, \ldots$ and $0 \leq k \leq 2^n - 2$;
  \item $\displaystyle{\lim_n \max_{0 \leq k < 2^n} \mu \bigl(D_{2^n+k}\bigr) = 0}$.
\end{enumerate}

Then set
$$
s_n = \widehat{[0,\alpha)} \sqcup \bigsqcup_{j=1}^{2^{n-1}} \widehat{D}_{2^n+2j-1}, \,\,\, n = 1,2, \ldots.
$$

Since $\alpha = \min D_{2^n}$, the union $\widehat{[0,\alpha)} \sqcup \widehat{D}_{2^n}$ is a union of dyadic intervals, and so $s_n \in \mathcal{B}_0$ for all $n \in \mathbb N$. By (1)-(2), $(s_n)_{n=1}^\infty$ is a pre-Rademacher family. First we prove (ii). Let $M \subseteq \mathbb N$ be an infinite subset. Let $0 \leq z \in \mathcal{B}_0$ be any lower bound for $\{s_n: \, n \in M\}$. Since $z \in \mathcal{B}_0$, one has that $z = \bigsqcup_{j \in J} \widehat{I}_n^j$ for suitable $n \in \mathbb N$ and $J \subseteq \{0,\ldots,2^n-1\}$. By (4), $I_n^j \subseteq [0,\alpha)$ for all $j \in J$. Since $\alpha$ is irrational and $J$ is finite, there exists a dyadic number $k/2^m$ with
$$
\max_{j \in J} \max I_n^j < \frac{k}{2^m} < \alpha.
$$
Thus, $z < \widehat{[0,k/2^m)}$ and $\widehat{[0,k/2^m)}$ is a lower bound for $\{s_n: \, n \in M\}$ in $\mathcal{B}_0$, and so (ii) is proved. Finally, (i) follows from (ii) and item (2) of Theorem~\ref{th:eqcard}.
\end{proof}

$\,\,\,\,\,\,$\textbf{Dyadic finitely additive measure on the algebra generated by a pre-Rademacher family.} \begin{thm} \label{thm:dyadicfam}
Let $\mathcal{R} = (r_i)_{i \in I}$ be a pre-Rademacher family in a Boolean algebra $\mathcal{B}$. Then there is a unique finitely additive measure $\mu: \mathcal{B}(\mathcal{R}) \to [0,1]$ satisfying \eqref{eq:main}.
\end{thm}

\begin{proof}
Let $\overline{\mathcal{P}}$ be the particle semialgebra with respect to $\mathcal{R}$. We prove that the map $\mu: \overline{\mathcal{P}} \to [0,1]$ defined by \eqref{eq:main} is a finitely additive measure (this will be enough by the well known fact \cite[Proposition~1.3.10]{Bog}). Let $J, J_n$ be finite subsets of $I$ for all $n \in N \subset \mathbb N$ and $\theta_j, \theta_{n,i} \in \{-1,1\}$ for all indices, where $N$ is finite. Assume
\begin{equation} \label{eq:zmnfn1}
\bigcap_{j \in J} \theta_j r_j = \bigsqcup_{n \in N} \bigcap_{j \in J_n} \theta_{n,j} r_j
\end{equation}
and prove
\begin{equation} \label{eq:zmnfn2}
\frac{1}{2^{|J|}} = \sum_{n \in N} \frac{1}{2^{|J_n|}}.
\end{equation}

Observe that
\begin{equation} \label{eq:zmnfn3}
J = \bigcap_{n \in N} J_n \,\,\, \mbox{and} \,\,\, \theta_j = \theta_{n,j} \,\,\, \mbox{for all} \,\,\, n \in N \,\,\, \mbox{and} \,\,\, j \in J.
\end{equation}

Indeed, for every given $j_0 \in J$ we intersect both parts of \eqref{eq:zmnfn1} by $\theta_{j_0} r_{j_0}$ and get
\begin{equation} \label{eq:zmnfn4}
\bigcap_{j \in J} \theta_j r_j = \bigsqcup_{n \in N} \Bigl( \bigcap_{j \in J_n} \theta_{n,j} r_j \cap \theta_{j_0} r_{j_0} \Bigr)
\end{equation}

Then the equality of the right hand sides of \eqref{eq:zmnfn1} and \eqref{eq:zmnfn4} gives that $\bigcap_{j \in J_n} \theta_{n,j} r_j = \bigcap_{j \in J_n} \theta_{n,j} r_j \cap \theta_{j_0} r_{j_0}$ for all $n \in N$, which yields that $j_0 \in J_n$ and $\theta_{j_0} = \theta_{n,j_0}$ for all $n \in N$. Thus, we have proved the inclusion $J \subseteq \bigcap_{n \in N} J_n$ and the desired equality of signs. It remains to show the converse inclusion. Let $j_0 \in \bigcap_{n \in N} J_n$. Intersect both parts of \eqref{eq:zmnfn1} by $\theta_{j_0} r_{j_0}$ and get
\begin{equation} \label{eq:zmnfn4}
\bigcap_{j \in J} \theta_j r_j \cap \theta_{j_0} r_{j_0} = \bigsqcup_{n \in N} \bigcap_{j \in J_n} \theta_{n,j} r_j
\end{equation}
which implies that $j_0 \in J$. Thus, \eqref{eq:zmnfn3} is proved.

Now we prove by induction in $m \in \mathbb N$ the following statement
\begin{enumerate}
  \item[$(T_m)$:] for every finite subsets $N \subset \mathbb N$, $J, J_n \subset I$, $n \in N$ such that $\Bigl| \bigcup_{n \in N} J_n \setminus J \Bigr| = m$, and every collection of signs $\theta_j, \theta_{n,i} \in \{-1,1\}$, \eqref{eq:zmnfn1} implies \eqref{eq:zmnfn2}.
\end{enumerate}

Note that the validity of $(T_m)$ for all $m = 0,1, \ldots$ means the finite additivity of $\mu$ on $\overline{\mathcal{P}}$.

If $m = 0$ then by \eqref{eq:zmnfn3}, $J = J_n$ for every $n \in N$, and hence $|N| = 1$ by the disjointness of the sum in \eqref{eq:zmnfn1}. This yields \eqref{eq:zmnfn2}.

Assume $(T_m)$ is true for $m = k$ and prove it for $m = k+1$. Fix any $i_0 \in \bigcup_{n \in N} J_n \setminus J$ (such an index exists because $k+1 > 0$). Now set
\begin{align*}
N' = \{n \in N: \,\, i_0 \notin J_n\}, \,\,\, &N^+ = \{n \in N: \,\, (i_0 \in J_n) \, \& (\theta_{n,i_0} = 1)\}, \\
&N^- = \{n \in N: \,\, (i_0 \in J_n) \, \& (\theta_{n,i_0} = -1)\}.
\end{align*}

Then for $p = \bigcap_{j \in J} \theta_j r_j$ one has
\begin{align} \label{eq:fnnjdjls81}
r_{i_0} \cap p = \Bigl( \bigsqcup_{n \in N'} r_{i_0} \cap \bigcap_{j \in J_n} \theta_{n,j} r_j \Bigr) \sqcup \Bigl( \bigsqcup_{n \in N^+} \bigcap_{j \in J_n} \theta_{n,j} r_j \Bigr);
\end{align}
and
\begin{align} \label{eq:fnnjdjls82}
-r_{i_0} \cap p = \Bigl( \bigsqcup_{n \in N'} -r_{i_0} \cap \bigcap_{j \in J_n} \theta_{n,j} r_j \Bigr) \sqcup \Bigl( \bigsqcup_{n \in N^-} \bigcap_{j \in J_n} \theta_{n,j} r_j \Bigr).
\end{align}

Since
$$
\Bigl| \Bigl( \{i_0\} \cup \bigcup_{n \in N' \cup N^+} \Bigr) \setminus \bigl( \{i_0\} \cup J \bigr)  \Bigr| \leq k \,\,\,\,\, \mbox{and} \,\,\,\,\, \Bigl| \Bigl( \{i_0\} \cup \bigcup_{n \in N' \cup N^-} \Bigr) \setminus \bigl( \{i_0\} \cup J \bigr)  \Bigr| \leq k,
$$
by the induction assumption, \eqref{eq:fnnjdjls81} and \eqref{eq:fnnjdjls81} imply
\begin{align} \label{eq:fnnjdjls83}
\frac{1}{2^{|J|+1}} = \sum_{n \in N'} \frac{1}{2^{|J_n|+1}} + \sum_{n \in N^+} \frac{1}{2^{|J_n|}}
\end{align}
and
\begin{align} \label{eq:fnnjdjls84}
\frac{1}{2^{|J|+1}} = \sum_{n \in N'} \frac{1}{2^{|J_n|+1}} + \sum_{n \in N^-} \frac{1}{2^{|J_n|}}.
\end{align}

Taking into account that $p = (r_{i_0} \cap p) \sqcup (-r_{i_0} \cap p)$, summing up \eqref{eq:fnnjdjls83} and \eqref{eq:fnnjdjls84} we obtain \eqref{eq:zmnfn2}. So, $(T_m)$ is proved for every $m$, and thus $\mu$ is finitely additive on $\overline{\mathcal{P}}$.
\end{proof}

\begin{defn} \label{def:dyadic}
Let $\mathcal{R} = (r_i)_{i \in I}$ be a pre-Rademacher family in a Boolean algebra $\mathcal{B}$. The finitely additive measure $\mu: \mathcal{B}(\mathcal{R}) \to [0,1]$ satisfying \eqref{eq:main} is called the \emph{dyadic measure} on the subalgebra $\mathcal{B}(\mathcal{R})$ generated by $\mathcal{R}$.
\end{defn}

At first glance, it is a striking fact that in the most natural cases the dyadic measure on $\mathcal{B}(\mathcal{R})$ is not countably additive (see the remark before the proof of Theorem~\ref{th:ccvmn740}).

$\,\,\,\,\,\,$\textbf{Maximal pre-Rademacher families.} \begin{defn}
A pre-Rademacher family $\mathcal{R}$ in a Boolean algebra $\mathcal{B}$ is said to be \emph{maximal} if there is no pre-Rademacher family in $\mathcal{B}$ including $\mathcal{R}$.
\end{defn}

Using Zorn's lemma, one can easily prove that every pre-Rademacher family can be extended to a maximal pre-Rademacher family. However it is surprisingly that the maximality is a bad property if one wants to define a measure by a pre-Rademacher family.

The following fact is not obvious.

\begin{thm} \label{th:kjhgf5}
The usual Rademacher family $(\widehat{r}_n)_{n=1}^\infty$ on $[0,1)$ is a complete pre-Rademacher family in the quotient Boolean algebra $\widehat{\mathcal{B}}$ modulo measure null sets of the Borel $\sigma$-algebra on $[0,1)$ failing to be maximal. Moreover, for every $\gamma \in (0,1)$ there exists an element $\widehat{r}_0 \in \widehat{\mathcal{B}}$ of measure $\mu(\widehat{r}_0) = \gamma$ such that $(\widehat{r}_n)_{n=0}^\infty$ is a pre-Rademacher family in $\widehat{\mathcal{B}}$.
\end{thm}

\begin{proof}
The fact that $(\widehat{r}_n)_{n=1}^\infty$ is a complete pre-Rademacher family in $\widehat{\mathcal{B}}$ is considered to be obvious. We show that it is not maximal. To prove the lack of maximality, we need a special terminology. For convenience of notation, we deal with the Borel $\sigma$-algebra $\mathcal{B}$ on $[0,1)$ and $(r_n)_{n \in \mathbb N}$. By $\mu$ we denote the Lebesgue measure on $\mathcal{B}$. We say that a set $A \in \mathcal{B}$ of positive measure is \emph{crushed} if $A$ includes no dyadic interval, up to a measure null set. In other words, a set $A$ is crushed if and only if $\mu(I \setminus A) > 0$ for every dyadic interval $I$ (which is the same as for every interval $I = (\alpha, \beta)$ with $0 \leq \alpha < \beta \leq 1$). Our goal is to construct a borel set $r_0 \subseteq [0,1)$ with $\mu(r_0) = \gamma$ such that both $r_0$ and $[0,1) \setminus r_0$ are crushed. Consequently, both $r_0$ and $[0,1) \setminus r_0$ intersect any dyadic interval by a set of positive measure. It would be enough, because the extended family $(r_n)_{n=0}^\infty$ is a pre-Rademacher family (observe that any particle with respect to $(r_n)_{n=0}^\infty$ includes some dyadic interval). We wonder if such a construction is already described in the literature, however knowing no citation we provide with a complete proof below.

Enumerate by $(I_m)_{m=1}^\infty$ all the dyadic intervals in the natural order $I_m = I_n^k$, where $m = 2^n+k-1$, $n = 0,1,2, \ldots$, $k = 1, \ldots, 2^n$. We say that a subset $A \subseteq [0,1)$ is a \emph{Cantor type set} $I$ if $A$ is crushed and the complement $[0,1) \setminus A$ is a finite or countable union of dyadic intervals. We remark that an at most countable union of dyadic intervals is also a disjoint union of dyadic intervals, as one can easily prove. Obviously, if $A$ is a Cantor type set, $I$ a dyadic interval and $A \subseteq I$ then the complement $I \setminus A$ is a finite or countable union of dyadic intervals. We need the following lemma.

\begin{lemma} \label{le:fut}
$\,$
\begin{enumerate}
  \item For any dyadic interval $I \subseteq [0,1)$ and any $\varepsilon > 0$ there exists a Cantor type set $A \subseteq I$ subset of measure $0 < \mu(A) \leq \varepsilon$.
  \item Let $A$ be a Cantor type set. Then for any dyadic interval $I$ the set $\mu(A \cap I) = 0$ is a Cantor type set.
  \item A finite disjoint union of Cantor type sets is a Cantor type set.
\end{enumerate}
\end{lemma}

\begin{proof}[Proof of Lemma~\ref{le:fut}]
(1) is easily proved by the well known construction of the Cantor set. (2) easily follows from the definitions. We prove (3). It is enough to prove for the union of two sets. So, let $A$ and $B$ be disjoint Cantor type sets. First we prove that $A \sqcup B$ is crushed. Assume, on the contrary, that this is not the case and there is a dyadic interval $I \subseteq A \sqcup B$, up to a measure null set $C$, that is, $I \subseteq A \sqcup B \sqcup C$. With no loss of generality we may and do assume that $I \subseteq A \sqcup B$, because the set $A \sqcup C$ is a Cantor type set as well.
Observe that $I = A' \sqcup B'$ where $A' = A \cap I$ and $B' = B \cap I$. By (2), $A'$ and $B'$ are Cantor type sets. So, there is a subset $J \subseteq \mathbb N$ such that $I \setminus A' = \bigsqcup_{j \in J} I_j$. On the other hand, $I \setminus A' = B'$. Hence, $B'$ contains dyadic intervals, which is impossible because $B'$ is crushed. Finally we show that $[0,1) \setminus (A \sqcup B)$ is an at most countable union of dyadic intervals. Indeed, let $[0,1) \setminus A = \bigsqcup_{i \in J_A} I_i$ and $[0,1) \setminus B = \bigsqcup_{j \in J_B} I_j$. Then
$$
[0,1) \setminus (A \sqcup B) = \bigl( [0,1) \setminus A \bigr) \cap \bigl( [0,1) \setminus B \bigr) = \Bigl( \bigsqcup_{i \in J_A} I_i \Bigr) \cap \Bigl( \bigsqcup_{j \in J_B} I_j \Bigr) = \bigsqcup_{i \in J_A} \bigsqcup_{j \in J_B} I_i \cap I_j.
$$
\end{proof}

Continue the proof of Theorem~\ref{th:kjhgf5}. We construct two disjoint sequences $(A_n)_{n=1}^\infty$ and $(B_n)_{n=1}^\infty$ of subsets of $[0,1)$ such that for all $n \in \mathbb N$
\begin{enumerate}
  \item $A_n$ and $B_n$ are Cantor type subsets of $I_n$;
  \item $A_i \cap B_j = \emptyset$ for all $i,j$;
  \item $\mu(A_n) \leq 2^{-n} \gamma$ and $\mu(B_n) \leq 2^{-n} (1 - \gamma)$.
\end{enumerate}

At the first step we choose any Cantor type set $A_1$ in $I_1$ with $0 < \mu(A_1) \leq \gamma/2$. Then $I_1 \setminus A_1 = \bigsqcup_{j \in J_1} I_j$ for some index set $J_1 \subseteq \mathbb N$, which is nonempty because $\mu(A_1) > 0$. Set $j_1 = \min J_1$ and choose any Cantor type set $B_1$ in $I_1$ with $0 < \mu(B_1) \leq (1 - \gamma)/2$.

Assume disjoint sets $A_1, \ldots, A_k$ and $B_1, \ldots, B_k$ are chosen to satisfy (1), (2), (3) for all $n = 1, \ldots, k$. By Lemma~\ref{le:fut} (3), $C = \bigsqcup_{n=1}^k A_k \sqcup \bigsqcup_{n=1}^k B_k$ is a Cantor type set, and by Lemma~\ref{le:fut} (2), $I_{k+1} \cap C$ is a Cantor type set as well. Choose any dyadic interval $I \subseteq I_{k+1} \setminus C$. As at the first step, choose disjoint Cantor type sets $A_{k+1}, B_{k+1} \subseteq I$ possessing (3) for $n = k+1$. The recursive construction is finished.

Now set $A_0 = \bigsqcup A_n$ and $B_0 = \bigsqcup B_n$. Then $\mu(A_0) = \sum_{n=1}^\infty \mu(A_n) \leq \sum_{n=1}^\infty 2^{-n} \gamma = \gamma$ and analogously $\mu(B_0) \leq 1 - \gamma$. Decompose $[0,1) \setminus (A_0 \sqcup B_0) = A' \sqcup B'$ onto measurable subsets of measures $\mu(A') = \gamma - \mu(A_0)$ and $\mu(B') = 1 - \gamma - \mu(B_0)$ and set $A = A_0 \sqcup A'$ and $B = B_0 \sqcup B'$. Then $[0,1) = A \sqcup B$ with $\mu(B) = \gamma$. Show that $A$ is crushed. Indeed, the inclusion $I_n \subseteq A$, up to a measure null set, is false, because $B_n \subseteq I_n \cap B_0 \subseteq I_n \cap B$. Analogously, $B$ is crushed. Setting $r = A$, we finish the proof.
\end{proof}

Remark that the constructed above extended pre-Rademacher family $(\widehat{r}_n)_{n=0}^\infty$ cannot define a countably additive measure on the Borel $\sigma$-algebra $\mathcal{\widehat{B}}$ by \eqref{eq:main} if $\gamma \neq 1/2$. Indeed, if such a measure $\widehat{\mu}$ existed, by the one hand, \eqref{eq:main} would imply that $\widehat{\mu}(\widehat{r}_0) = 1/2$. But on the other hand, $\widehat{\mu}$ must coincide with the Lebesgue measure on $\mathcal{\widehat{B}}$ because both measures have the same values at dyadic intervals. Hence, $\widehat{\mu}(\widehat{r}_0) = \gamma$, a contradiction.

Using the scheme of the proof of Theorem~\ref{th:kjhgf5}, one can prove the following general fact.

\begin{thm} \label{th:kjhgf7}
Let $\mathcal{B}$ be an atomless Boolean Dedekind $\sigma$-complete Boolean algebra. Then any countable pre-Rademacher family in $\mathcal{B}$ is not maximal.
\end{thm}

\section{Rademacher families} \label{sec:Radfam}

The following property of a pre-Rademacher family gives a kind of the ``\textbf{independence}'' condition which is needed to be a Rademacher family.

\begin{defn}
A pre-Rademacher family $\mathcal{R}$ in a Boolean algebra $\mathcal{B}$ is called \emph{minimal} if for any $r \in \mathcal{R}$ one has that $r \notin \mathcal{B}_\tau(\mathcal{R} \setminus \{r\})$.
\end{defn}

Simple examples (like a disjoint family of nonzero elements) show that a minimal family need not be pre-Rademacher. On the other hand, by the results of the previous section, there are pre-Rademacher families which are not minimal.

Another property of a pre-Rademacher family that is needed to make it a Rademacher family is an analogue of the property to be ``\textbf{identically distributed}''.

\begin{defn}
An infinite pre-Rademacher family $(r_i)_{i \in I}$ in a Boolean algebra $\mathcal{B}$ is called:
\begin{itemize}
  \item \emph{hereditarily atomless} if every of its infinite subfamily is atomless;
  \item \emph{vanishing at infinity} if $\bigcap_{j \in J} \theta_j r_j = \mathbf{0}$ for any infinite subset $J \subseteq I$ and any collection of signs $\theta_j = \pm 1$, $j \in J$.
\end{itemize}
\end{defn}

The following statement is a consequence of Theorem~\ref{th:eqcard}.

\begin{prop} \label{pr:cons}
Let $\mathcal{B}$ be a Boolean algebra. Then the following assertions hold.
\begin{enumerate}
  \item Every vanishing at infinity pre-Rademacher family in $\mathcal{B}$ is hereditarily atomless.
  \item If, moreover, $\mathcal{B}$ is $\sigma$-complete then the converse also holds: an infinite pre-Rademacher family in $\mathcal{B}$ is hereditarily atomless if and only if it vanishes at infinity.
\end{enumerate}
\end{prop}

As Example~\ref{ex:ssjhd} shows, a hereditarily atomless pre-Rademacher family need not be vanishing at infinite. So, the $\sigma$-completeness assumption in (2) of Proposition~\ref{pr:cons} is essential.

\begin{defn}
A vanishing at infinity minimal pre-Rademacher family in a Boolean algebra $\mathcal{B}$ is called a \emph{Rademacher family} in $\mathcal{B}$.
\end{defn}

To emphasize the importance of the vanishing at infinity property of a Rademacher family, we provide examples showing the variety of distinct (non-isomorphic) types of countable minimal pre-Rademacher families without this property. Another sense of the vanishing at infinity property is explained below in Theorem~\ref{th:ccvmn740}: this property gives positivity of the measure generated by a pre-Rademacher family.

\begin{prop} \label{ex:5eqcardex}
Let $\mathcal{B}$ be a purely atomic $\tau$-complete Boolean algebra with the set $\mathcal{A}_0$ of atoms of cardinality $\aleph_0 \leq |\mathcal{A}_0| \leq c$, where $c$ is the cardinality of continuum. Then there exists a countable $\sigma$-complete (and hence, $\tau$-complete) minimal pre-Rademacher family in $\mathcal{B}$.
\end{prop}

\begin{proof}
With no loss of generality we assume that $\mathcal{A}_0 = \bigl\{ \{a\}: \, a \in \mathcal{B}_0 \bigr\}$, where $\mathcal{B}_0$ is a dense subset of $[0,1)$ and $\mathcal{B}$ is the power set of $\mathcal{B}_0$, that is, the set of all subsets of $\mathcal{B}_0$. Any number $x \in [0,1)$ we represent as $x = \sum_{n=1}^\infty a_n(x) \, 2^{-n}$, where the dyadic digits $a_n(x) \in \{0,1\}$ are not eventually $1$'s. We set $r_n = \{x \in \mathcal{B}_0: \, a_n(x) = 1\}$. Then for any finite collection of distinct numbers $n_1, \ldots, n_k \in \mathbb N$ and signs $\theta_1, \ldots, \theta_k = \pm 1$ one has
$$
\theta_1 r_{n_1} \cap \ldots \cap \theta_k r_{n_k} = \Bigl\{ x \in [0,1): \, (\forall i \leq k) \, a_{n_i}(x) = \frac{\theta_i + 1}{2} \Bigr\} \cap \mathcal{B}_0,
$$
which is nonempty because $\mathcal{B}_0$ is dense in $[0,1)$. To show that $(r_n)_{n=1}^\infty$ is $\sigma$-complete, observe that for any $y \in [0,1)$ one has
$$
\{y\} = \bigcap_{n=1}^\infty \Bigl\{ x \in [0,1): \, a_n(x) = a_n(y) \Bigr\} = \bigcap_{n=1}^\infty \bigl( 2 a_n(y) - 1 \bigr) \Bigl\{ x \in [0,1): \, a_n(x) = 1 \Bigr\}.
$$
In particular, for any $y \in \mathcal{B}_0$ one has $\{y\} = \bigcap_{n=1}^\infty \bigl( 2 a_n(y) - 1 \bigr) \, r_n.$

Finally, we show that $(r_n)_{n=1}^\infty$ is minimal. Fix any $n_0 \in \mathbb N$. We prove the following claim.

\textbf{Claim.} For any $A \in \mathcal{B}_\tau \bigl( \{r_n: \, n \neq n_0\} \bigr)$ one has
\begin{equation} \label{eq:fzfjdh}
(\forall x \in A) \,\,\, x^* \overset{ \text{def} }{=} \sum_{n \neq n_0} a_n(x) 2^{-n} + \bigl( 1 - a_{n_0}(x) \bigr) \, 2^{-n_0} \in A.
\end{equation}
First observe that \eqref{eq:fzfjdh} holds for $A = r_n$ with any $n \neq n_0$. Hence, by (1) of Proposition~\ref{pr:semigen}, \eqref{eq:fzfjdh} holds for all $A \in \mathcal{B} \bigl( \{r_n: \, n \neq n_0\} \bigr)$. Now fix any $A' \in \mathcal{B}_\tau \bigl( \{r_n: \, n \neq n_0\} \bigr)$. By (2) of Proposition~\ref{pr:semigen}, there exists a net $(A_\alpha)$ in $\mathcal{B} \bigl( \{r_n: \, n \neq n_0\} \bigr)$ and a net $(u_\alpha)$ in $\mathcal{B}$ with the same index set such that $u_\alpha \downarrow \mathbf{0}$ and $A_\alpha \triangle A' \leq u_\alpha$. Hence, given any $x \in A'$, there exists $\alpha_0$ such that $x \notin u_\alpha$, and therefore, $x \in A_\alpha$ for all $\alpha \geq \alpha_0$. Since \eqref{eq:fzfjdh} holds for $A = A_\alpha$, we have that $X^* \in A_\alpha$ for all $\alpha \geq \alpha_0$, and so $x^* \in A$. Thus, \eqref{eq:fzfjdh} holds for $A = A'$ and the claim is proved.

Since \eqref{eq:fzfjdh} does not hold for $A = r_{n_0}$, we deduce that $r_{n_0} \in \mathcal{B}_\tau(\mathcal{R}) \setminus \mathcal{B}_\tau \bigl( \{r_n: \, n \neq n_0\} \bigr)$.
\end{proof}

\begin{thm} \label{thm:5eqcardex}
Let $\mathcal{B} = \mathcal{B}_1 \oplus \mathcal{B}_2$ be the direct sum of a Boolean algebra $\mathcal{B}_1$ isomorphic to $\widehat{\mathcal{B}}$ and a purely atomic $\tau$-complete Boolean algebra $\mathcal{B}_2$ with the set $\mathcal{A}_0$ of atoms of cardinality $|\mathcal{A}_0| \leq c$, where $c$ is the cardinality of continuum. Then there exists a countable $\sigma$-complete (and hence, $\tau$-complete) minimal pre-Rademacher family in $\mathcal{B}$.
\end{thm}

For the proof, we need some lemmas which will be useful also for further consideration.

We say that a family $\mathcal{R} = (r_i)_{i \in I}$ of elements of a Boolean algebra $\mathcal{B}$ is a \emph{pre-Rademacher family on an element} $e \in \mathcal{B} \setminus \{\mathbf{0}\}$ if $\mathcal{R}$ is a pre-Rademacher family in the Boolean algebra $\mathcal{B}_e = \{x \in \mathcal{B}: \, x \leq e\}$ with the operations induced by $\mathbf{B}$ and unity $e$. We also introduce the notions of a \emph{Rademacher family, minimal pre-Rademacher family, hereditarily atomless pre-Rademacher family, vanishing at infinity pre-Rademacher family, complete Rademacher and complete pre-Rademacher family on an element} $e$ similarly.

\begin{lemma} \label{le:obvcnjfu}
Let $\mathcal{R} = (r_i)_{i \in I}$ be a family of elements of a Boolean algebra $\mathcal{B}$ and $e$ is an upper bound for $\mathcal{R}$ in $\mathcal{B}$. Then for the assertions
\begin{enumerate}
  \item[(1)] $\mathcal{R}$ is a pre-Rademacher family on $e$;
  \item[(2)] $\mathcal{R}$ is a pre-Rademacher family;
\end{enumerate}
one has (1) $\Rightarrow$ (2) anyway, and (2) $\Rightarrow$ (1) if, in addition, $\mathcal{R}$ is infinite.
\end{lemma}

Before the proof, we notice that the implication (2) $\Rightarrow$ (1) is generally false for finite families $\mathcal{R} = (r_i)_{i \in I}$. Consider, for example, the family $\mathcal{R} = \{r_1, r_2\}$ in $\widehat{\mathcal{B}}$, where $r_1 = \widehat{\bigl[0,\frac12\bigr)}$ and $r_2 = \widehat{\bigl[\frac14,\frac34\bigr)}$. Then $\mathcal{R}$ is a pre-Rademacher family on $\widehat{\mathcal{B}}$, however is not a pre-Rademacher family on $e = \widehat{\bigl[0,\frac34\bigr)}$ (which is an upper bound for $\mathcal{R}$), because $(-r_1 \cap e) \cap (-r_2 \cap e) = \mathbf{0}$.

\begin{proof}
The implication (1) $\Rightarrow$ (2) is obvious. We prove (2) $\Rightarrow$ (1) for an infinite $\mathcal{R}$. Assume, on the contrary, that (2) holds, however there exists a finite subset $J \subset I$ and a collection of signs $\theta_j = \pm 1$, $j \in J$ such that $\bigcap_{j \in J} (\theta_j r_j \cap e) = \mathbf{0}$. Since $I$ is infinite, there exists $i \in I \setminus J$. Then
$$
r_i \cap \bigcap_{j \in J} \theta_j r_j \leq e \cap \bigcap_{j \in J} \theta_j r_j = \bigcap_{j \in J} (\theta_j r_j \cap e) = \mathbf{0},
$$
which contradicts (2).
\end{proof}

\begin{lemma} \label{le:obvcnjfu2}
Let $\mathcal{R} = (r_i)_{i \in I}$ be a family of elements of a Boolean algebra $\mathcal{B}$, $e$ is an upper bound for $\mathcal{R}$ in $\mathcal{B}$ such that $\mathcal{R}$ is a pre-Rademacher family on $e$. Then the following assertions are equivalent:
\begin{enumerate}
  \item[(1)] $\mathcal{R}$ is minimal as a pre-Rademacher family in $\mathcal{B}$;
  \item[(2)] $\mathcal{R}$ is a minimal pre-Rademacher family on $e$;
\end{enumerate}
\end{lemma}

The proof is an easy exercise, due to Proposition~\ref{pr:semigen}.

Remark that a similar statement on a family to be Rademacher in $\mathcal{B}$ and on an upper bound of the family is not true, as well as for the vanishing at infinity and hereditarily atomless properties of a pre-Rademacher family.

The following two lemmas constitute an essential step in the procedure of gluing together separate pre-Rademacher families.

\begin{lemma} \label{le:strangeexhaust}
Let $\mathcal{B}$ be a Boolean algebra, $0 < e \in \mathcal{B}$ and $e = e' \sqcup e''$ with $e', e'' >0$. Let $\mathcal{R}' = (r_i')_{i \in I}$ and $\mathcal{R}'' = (r_i'')_{i \in I}$ be pre-Rademacher families on $e'$ and $e''$ respectively. Then $e' \notin \mathcal{B}_\tau(\mathcal{R})$, where $\mathcal{R} = (r_i)_{i \in I}$ and $r_i = r_i' \sqcup r_i''$ for every $i \in I$.
\end{lemma}

\begin{proof}
Assume, on the contrary, that there exist a net $(z_\alpha)$ in $\mathcal{B}(\mathcal{R})$ and a net $(w_\alpha)$ in $\mathcal{B}$ with the same index set such that $z_\alpha \triangle e' \leq w_\alpha \downarrow \mathbf{0}$. Say, $z_\alpha = \bigsqcup_{k=1}^{m_\alpha} s_{\alpha, k}$, where $s_{\alpha,k} = \bigcap_{j \in J_{\alpha,k}} \theta_{\alpha,k,j} r_j$, $\theta_{\alpha,k,j} = \pm 1$ and $J_{\alpha,k}$ are finite subsets of $I$. Set
$$
x_\alpha = z_\alpha \cap e', \, y_\alpha = z_\alpha \cap e'', \, u_\alpha = w_\alpha \cap e', \, v_\alpha = w_\alpha \cap e'', \, p_{\alpha,k} = s_{\alpha,k} \cap e', \, q_{\alpha,k} = s_{\alpha,k} \cap e''
$$
and observe that
$$
x_\alpha = \bigsqcup_{k=1}^{m_\alpha} p_{\alpha, k}, \,\,\, y_\alpha = \bigsqcup_{k=1}^{m_\alpha} q_{\alpha, k}, \,\,\, p_{\alpha,k} = \bigcap_{j \in J_{\alpha,k}} \theta_{\alpha,k,j} r_j', \,\,\, q_{\alpha,k} = \bigcap_{j \in J_{\alpha,k}} \theta_{\alpha,k,j} r_j'',
$$
$e' \setminus x_\alpha \leq u_\alpha \downarrow \mathbf{0}$ and $y_\alpha \leq v_\alpha \downarrow \mathbf{0}$. Now we define a function $\varphi: \mathcal{B}_e(\mathcal{R}_1) \to \mathcal{B}_e(\mathcal{R}_1)$ by setting
$$
\varphi \Bigl( \bigsqcup_{i=1}^n \bigcap_{j \in J_i} \theta_{i,j} r_j' \Bigr) = \bigsqcup_{i=1}^n \bigcap_{j \in J_i} \theta_{i,j} r_j''
$$
for all $n \in \mathbb N$, finite subsets $J_i$ of $I$ and signs $\theta_{i,j}$. Since $\mathcal{R}'$ and $\mathcal{R}''$ are pre-Rademacher families, $\varphi$ is an isomorphism. Choose any index $\alpha_0$ so that $v_{\alpha_0} < e''$ (such an index exists, because $v_\alpha \downarrow \mathbf{0}$). Since $x_\alpha \geq e' \setminus u_\alpha \uparrow e'$, one has that $\sup_{\alpha \geq \alpha_0} x_\alpha = e'$. Since $\varphi$ is an isomorphism and $\varphi(x_\alpha) = y_\alpha$, we obtain that $\sup_{\alpha \geq \alpha_0} y_\alpha = e''$, which contradicts the choice of $\alpha_0$, because $y_\alpha \leq v_\alpha \leq v_{\alpha_0} < e''$ for all $\alpha \geq \alpha_0$.
\end{proof}

\begin{lemma} \label{le:collecting}
Let $\mathcal{B}$ be a Boolean algebra, $0 < e \in \mathcal{B}$ and $e = e' \sqcup e''$ with $e', e'' >0$. Let $(r_i')_{i \in I}$ and $(r_i'')_{i \in I}$ be pre-Rademacher families on $e'$ and $e''$ respectively. Then $(r_i)_{i \in I_0}$ is a pre-Rademacher family on $e$, where $I_0 = I \cup \{I\}$, $r_i = r_i' \sqcup r_i''$ for $i \in I$ and $r_I = e'$. If, in addition, $(r_i')_{i \in I}$ and $(r_i'')_{i \in I}$ are
\begin{enumerate}
  \item minimal
  \item vanishing at infinity
  \item $\tau$-complete
  \item $\sigma$-complete
\end{enumerate}
on $e'$ and $e''$ respectively then so is $(r_i)_{i \in I_0}$ on $e$.
\end{lemma}

\begin{proof}
First we show that $(r_i)_{i \in I_0}$ is a pre-Rademacher family on $e$. Let $J \subseteq I_0$ be a finite subset and $\theta_j = \pm 1$ for all $j \in J$. If $I \notin J$ then in view of $r_i' \cap r_j'' = \mathbf{0}$ we obtain
$$
\bigcap_{j \in J} \theta_j r_j = \Bigl(\bigcap_{j \in J} \theta_j r_j'\Bigr) \sqcup \Bigl(\bigcap_{j \in J} \theta_j r_j''\Bigr) > \mathbf{0}.
$$
If $I \in J$ and $\theta_I = 1$ then
$$
\bigcap_{j \in J} \theta_j r_j = e' \cap \Bigl(\bigcap_{j \in J \setminus \{I\}} \theta_j r_j \Bigr) = \bigcap_{j \in J \setminus \{I\}} \theta_j r_j' > \mathbf{0}.
$$
Analogously, if $I \in J$ and $\theta_I = -1$ then $\bigcap_{j \in J} \theta_j r_j = \bigcap_{j \in J \setminus \{I\}} \theta_j r_j'' > \mathbf{0}.$ Thus, all finite particles are nonzero and hence, $(r_i)_{i \in I_0}$ is a pre-Rademacher family on $e$.

(1) The condition $r_I \notin \mathcal{B}_\tau\bigl((r_i)_{i \in I}\bigr)$ follows from Lemma~\ref{le:strangeexhaust}. Show that for any $i_0 \in I$ one has $r_{i_0} \notin \mathcal{B}_\tau\bigl((r_i)_{i \in I_0 \setminus \{i_0\}}\bigr)$. Assume, on the contrary, that there exists a net $(z_\alpha)$ in $\mathcal{B}\bigl((r_i)_{i \in I_0 \setminus \{i_0\}}\bigr)$ such that $z_\alpha \overset{\rm o}\longrightarrow r_{i_0}$. Then
\begin{equation} \label{eq:cmfnndi}
z_\alpha \cap e' \overset{\rm o}\longrightarrow r_{i_0} \cap e' = r_{i_0}'.
\end{equation}
Since $z_\alpha \cap e' \in \mathcal{B}\bigl((r_i')_{i \in I \setminus \{i_0\}}\bigr)$, \eqref{eq:cmfnndi} implies $r_{i_0}' \in \mathcal{B}_\tau\bigl((r_i')_{i \in I \setminus \{i_0\}}\bigr)$, which contradicts the assumption.

Let $J \subseteq I_0$ be an infinite subset and $\theta_j = \pm 1$ for all $j \in J$. Then
\begin{align*}
\bigcap_{j \in J} (\theta_j r_j)^+ &\leq \bigcap_{j \in J \setminus \{I\}} (\theta_j r_j)^+ = \bigcap_{j \in J \setminus \{I\}} (\theta_j r_j')^+ + \bigcap_{j \in J \setminus \{I\}} (\theta_j r_j'')^+ = 0 + 0 = 0.
\end{align*}
So, all infinite particles are zero, and thus $(r_i)_{i \in I_0}$ is a Rademacher family on $e$.

(2) Let $J \subseteq I_0$ be an infinite subset and $(\theta_j)_{j \in J}$ a collection of signs. If $I \in I_0$ and $\theta_I = 1$ then $\bigcap_{j \in J} \theta_j r_j = \bigcap_{j \in J \setminus \{I\}} \theta_j r_j' = \mathbf{0}$, because $(r_i')_{i \in I}$ vanishes at infinity and $J \setminus \{I\}$ is infinite. Analogously, if $I \in I_0$ and $\theta_I = -1$ then $\bigcap_{j \in J} \theta_j r_j = \bigcap_{j \in J \setminus \{I\}} \theta_j r_j'' = \mathbf{0}$. Finally, if $I \notin I_0$ then $J \subseteq I$ and hence
$$
\bigcap_{j \in J} \theta_j r_j = \bigcap_{j \in J} \theta_j r_j' \sqcup \bigcap_{j \in J} \theta_j r_j'' = \mathbf{0} \sqcup \mathbf{0} = \mathbf{0}.
$$

(3) Given any $z \in \mathcal{B}_e$, we split $z = x \sqcup y$, where $x = z \cap e'$ and $y = z \cap e''$. Using the assumption, choose a net $(x_\alpha)_{\alpha \in A}$ in $\mathcal{B} \bigl( (r_i')_{i \in I} \bigr)$, a net $(u_\alpha)_{\alpha \in A}$ in $\mathcal{B}$ with $x_\alpha \triangle x \leq u_\alpha \downarrow \mathbf{0}$, a net $(y_\beta)_{\alpha \in B}$ in $\mathcal{B} \bigl( (r_i'')_{i \in I} \bigr)$ and a net $(v_\beta)_{\alpha \in B}$ in $\mathcal{B}$ with $y_\beta \triangle y \leq v_\beta \downarrow \mathbf{0}$. Say,
$$
x_\alpha = \bigsqcup_{k=1}^{m_\alpha} p_{\alpha,k}, \,\, p_{\alpha,k} = \bigcap_{j \in J_{\alpha,k}'} \theta_{\alpha,k,j} r_j' \,\,\,\,\, \mbox{and} \,\,\,\,\, y_\beta = \bigsqcup_{k=1}^{n_\beta} q_{\beta,k}, \,\, q_{\beta,k} = \bigcap_{j \in J_{\beta,k}''} \theta_{\beta,k,j} r_j'',
$$
where $J_{\alpha,k}'$ and $J_{\beta,k}''$ are finite subsets of $I$ and $\theta_{\alpha,k,j}$, $\theta_{\beta,k,j}$ are signs. Observe that $C = A \times B$ is a directed set with respect to the order $(\alpha',\beta') \leq (\alpha'', \beta'')$ if and only if $\alpha' \leq \alpha''$ and $\beta' \leq \beta''$. We set $z_{(\alpha,\beta)} = x_\alpha \sqcup y_\beta$ and prove that $z_{(\alpha,\beta)} \stackrel{\rm o}{\longrightarrow} z$ and $z_{(\alpha,\beta)} \in \mathcal{B} \bigl( (r_i)_{i \in I_0} \bigr)$. Indeed,
$$
z_{(\alpha,\beta)} \triangle (u_\alpha \sqcup v_\beta) = (x_\alpha \triangle x) \sqcup (y_\beta \triangle y) \leq u_\alpha \sqcup v_\beta \overset{\text{def}}{=} w_{(\alpha,\beta)} \downarrow \mathbf{0}.
$$

It remains to observe that
$$
p_{\alpha,k} = r_I \cap \bigcap_{j \in J_{\alpha,k}'} \theta_{\alpha,k,j} r_j \,\,\,\,\, \mbox{and} \,\,\,\,\, q_{\beta,k} = -r_I \cap \bigcap_{j \in J_{\beta,k}''} \theta_{\beta,k,j} r_j
$$
are particles with respect to $(r_i)_{i \in I_0}$.

(4) The proof is similar to that of (3).
\end{proof}

\begin{proof}[Proof of Theorem~\ref{thm:5eqcardex}]
For simplicity of notation, with no loss of generality we assume that $\mathcal{B}_1 = \widehat{\mathcal{B}}$. We consider two cases depending on whether $\mathcal{A}_0$ is finite or infinite. Let first $\mathcal{A}_0$ be finite, say, $\mathcal{A}_0 = \{a_1, \ldots, a_m\}$ with $a_i \neq a_j$ for $i \neq j$. Decompose $\mathbb N = N_1 \sqcup \ldots \sqcup N_m$ onto infinite subsets $N_k$. Define a function $\phi: \mathbb N \to \{1, \ldots, m\}$ so that $n \in N_{\phi(n)}$ and set $r_n = (\widehat{r}_n,a_{\phi(n)}) \in \mathcal{B}$ for all $n \in \mathbb N$. We prove that $(r_n)_{n=1}^\infty$ possesses the desired properties. If $J \subseteq \mathbb N$ is a finite subsets and $\theta_j = \pm 1$, $j \in J$ then
$$
\bigcap_{j \in J} \theta_j r_j = \Bigl( \bigcap_{j \in J} \theta_j \widehat{r}_j, \bigcap_{j \in J} \theta_j a_{k(n)} \Bigr) \neq \mathbf{0}
$$
because $\bigcap_{j \in J} \theta_j \widehat{r}_j \neq \mathbf{0}$. So, $(r_n)_{n=1}^\infty$ is a pre-Rademacher family. A similar coordinate-wise argument shows that $(r_n)_{n=1}^\infty$ is minimal. We prove that $(r_n)_{n=1}^\infty$ is complete. Let $z = (x,y) \in \mathcal{B}$ be arbitrary. Choose a sequence $(x_n)_{n=1}^\infty$ in $\mathcal{B} \bigl( (\widehat{r}_n)_{n=1}^\infty\bigr)$ such that $x_n \stackrel{\rm o}{\longrightarrow} x$. Say,
$$
x_n = \bigsqcup_{k=1}^{m_n} \bigcap_{j \in J_{n,k}} \theta_{n,k,j} \widehat{r}_j,
$$
where $\theta_{n,k,j} = \pm 1$ and $J_{n,k} \subseteq \mathbb N$ are finite subsets.

Given any $t \in \mathcal{B}_2$, by $I(t)$ we denote the subset of $\{1, \ldots, m\}$ such that $t = \bigsqcup_{i \in I(t)} a_i$. Then for every $t \in \mathcal{B}_2$ one has
\begin{equation} \label{eq:ajshf}
\bigcup_{i \in I(t)} \bigcap_{j \in N_{\phi(i)}} r_j = \bigcup_{i \in I(t)} \Bigl( \bigcap_{j \in N_{\phi(i)}} \widehat{r}_j, \bigcap_{j \in N_{\phi(i)}} a_i \Bigr) = \bigcup_{i \in I(t)} (\mathbf{0}, a_i) = (\mathbf{0},t).
\end{equation}

Set
$$
y_n =  \bigsqcup_{k=1}^{m_n} \bigcap_{j \in J_{n,k}} \theta_{n,k,j} a_{\phi(n)} \,\,\,\,\, \mbox{and} \,\,\,\,\, z_n = (x_n,y).
$$

Then $z_n = \bigl((x_n,y_n) \cup (\mathbf{0},y) \bigr) \setminus (\mathbf{0},y_n)$ and $z_n \stackrel{\rm o}{\longrightarrow} z$. By \eqref{eq:ajshf}, $(\mathbf{0},y), (\mathbf{0},y_n) \in \mathcal{B}_\sigma \bigl( (r_j)_{j=1}^\infty \bigr)$. Taking into account that $(x_n,y_n) = \bigsqcup_{k=1}^{m_n} \bigcap_{j \in J_{n,k}} \theta_{n,k,j} r_j$, we deduce that $z_n \in \mathcal{B}_\sigma \bigl( (r_j)_{j=1}^\infty \bigr)$. Thus, the pre-Rademacher family $(r_j)_{j=1}^\infty$ is complete.
\end{proof}

Remark that the property of a family to be Rademacher is weaker than the property to be independent in the sense of probability, which depends on a measure. More precisely, a Rademacher family is independent in the sense of probability with respect to a suitable measure, as we see later.

\begin{thm} \label{th:ccvmn740}
Let $\mathcal{R} = (r_i)_{i \in I}$ be a minimal pre-Rademacher family in a Boolean algebra $\mathcal{B}$. Then there is a unique countably additive probability measure $\mu$ on $\mathcal{B}_\sigma(\mathcal{R})$ satisfying \eqref{eq:main}. If, moreover, $\mathcal{R}$ vanishes at infinity (in other words, is a Rademacher family) then $\mu$ is positive.
\end{thm}

Before the proof, we remark that the usual way to start the proof with showing that the dyadic finitely additive measure $\mu$ on $\mathcal{B}(\mathcal{R})$ the existence and uniqueness of which is already established in Theorem~\ref{thm:dyadicfam} is countably additive is false. Moreover, $\mu$ is not countably additive on $\mathcal{B}(\mathcal{R})$ for natural examples. Indeed, we show that the restriction $\mu_0 = \mu |_{\widehat{\mathcal{B}}(\widehat{\mathcal{R}})}$ of the Lebesgue measure $\mu$, which is countably additive on $\widehat{\mathcal{B}}$, to the subalgebra $\widehat{\mathcal{B}}(\widehat{\mathcal{R}})$ generated by the usual Rademacher family $\widehat{\mathcal{R}}$, is not countably additive. To do this, we provide an example of a sequence $(x_n)_{n=1}^\infty$ in $\widehat{\mathcal{B}}(\widehat{\mathcal{R}})$ with $x_{n+1} \leq x_n$ for all $n$, $\inf_n x_n = 0$ in $\widehat{\mathcal{B}}(\widehat{\mathcal{R}})$ (but not in $\widehat{\mathcal{B}}$!) with $\mu(x_n) \geq 1/2$ for all $n$, which is enough by the well known and easily proved fact \cite[326~F(c)]{F3}. Let $(I_n)_{n=1}^\infty$ be any numeration of the dyadic intervals. For all $n \in \mathbb N$ choose a dyadic interval $I_{k_n} \subseteq I_n$ of measure $\mu(I_{k_n}) \leq 2^{-n-1}$ and set $x_m = [0,1) \setminus \bigcup_{n=1}^m I_{k_n}$. Then $x_{n+1} \leq x_n$ for all $n$ and
$$
\mu(x_m) \geq 1 - \sum{n=1}^m \mu(I_{k_n}) \geq 1 - \sum{n=1}^m \frac{1}{2^{n+1}} > 1 - \frac12 = \frac12.
$$

Prove that $\inf_n x_n = 0$ in $\widehat{\mathcal{B}}(\widehat{\mathcal{R}})$. Let $z \in \widehat{\mathcal{B}}(\widehat{\mathcal{R}})$ be an upper bound for $\{x_n: \, n \in \mathbb N\}$. Assume on the contrary that $z > 0$. Then there exists a dyadic interval $I_m$ such that $I_m \leq z$, and hence $I_{k_m} \subseteq I_m \leq z \leq x_m$, which contradicts the choice of $x_m$.

We need two technical lemmas.

\begin{lemma} \label{le:tech1}
Let $\mathcal{R} = (r_i)_{i \in I}$ be a pre-Rademacher system in a Boolean algebra $\mathcal{B}$. Then any element $x \in \mathcal{B}(\mathcal{R})$ has an expansion
\begin{equation} \label{eq:aev85}
x = \bigsqcup_{k=1}^m \bigcap_{j \in J} \theta_{k,j} r_j,
\end{equation}
where $m \in \mathbb N$, $J$ is a finite subset of $I$ and $\theta_{k,j} = \pm 1$.
\end{lemma}

\begin{proof}
Assume $x \in \mathcal{B}(\mathcal{R})$, say,
\begin{equation} \label{eq:aev86}
x = \bigsqcup_{i=1}^n \bigcap_{j \in J_i} \theta_{i,j} r_j,
\end{equation}
where $n \in \mathbb N$, $J_i$ are finite subsets of $I$ and $\theta_{i,j} = \pm 1$. Set $J = \bigcup_{i=1}^n J_k$. Then by \eqref{eq:beg}, for every $i = 1, \ldots, n$ one has
\begin{equation} \label{eq:aev87}
\bigcap_{j \in J_i} \theta_{i,j} r_j = \bigsqcup_{\epsilon_j = \pm 1} \bigcap_{j \in J_i} \theta_{i,j} r_j \cap \bigcap_{j \in J \setminus J_i} \epsilon_j r_j,
\end{equation}
where the disjoint sum is taken over all possible collections of signs $(\epsilon_j)_{j \in J \setminus J_i}$. Substituting \eqref{eq:aev87} to \eqref{eq:aev86} we obtain \eqref{eq:aev85}.
\end{proof}

\begin{lemma} \label{le:tech2}
Let $\mathcal{R'} = (r_i')_{i \in I}$ and $\mathcal{R'} = (r_i')_{i \in I}$ be pre-Rademacher system in Boolean algebra $\mathcal{B}'$ and $\mathcal{B}'$ respectively. Then the function $\varphi: \mathcal{B}'(\mathcal{R}') \to \mathcal{B}''(\mathcal{R}'')$ defined by
$$
\varphi \Bigl( \bigsqcup_{k=1}^m \bigcap_{j \in J} \theta_{k,j} r_j' \Bigr) = \bigsqcup_{k=1}^m \bigcap_{j \in J} \theta_{k,j} r_j''
$$
for all $m \in \mathbb N$, finite subsets $J$ of $I$ and signs $\theta_{k,j} = \pm 1$, is order preserving, that is, for every $x,y \in \mathcal{B}'(\mathcal{R}')$ the inequality $x \leq y$ implies $\varphi(x) \leq \varphi(y)$.
\end{lemma}

\begin{proof}
Observe that for any finite subsets $J', J'' \subset I$ and any signs $\theta_j' = \pm 1$, $j \in J'$ and $\theta_j'' = \pm 1$, $j \in J''$ the following inclusions are equivalent
\begin{equation} \label{eq:adjn1}
\bigcap_{j \in J'} \theta_j' r_j' \subseteq \bigcap_{j \in J''} \theta_j'' r_j'
\end{equation}
and
\begin{equation} \label{eq:adjn2}
\bigcap_{j \in J'} \theta_j' r_j'' \subseteq \bigcap_{j \in J''} \theta_j'' r_j'',
\end{equation}
because each of the conditions is equivalent to the next one: $J'' \subseteq J'$ and $\theta_j'' = \theta_j'$ for all $j \in J''$. Using the equivalence of \eqref{eq:adjn1} and \eqref{eq:adjn2}, one can easily deduce the lemma statement.
\end{proof}

\begin{proof}[Proof of Theorem~\ref{th:ccvmn740}]
If $I$ is finite then the theorem statement is obvious. Let $I$ be infinite and $\Sigma$ be an algebra of subsets of a set $\Omega$ isomorphic to $\mathcal{B}_\sigma(\mathcal{R})$ (such a set and an algebra exist by the Stone representation theorem) and let $\varphi: \mathcal{B}_\sigma(\mathcal{R}) \to \Sigma$ be an isomorphism. Let $\Sigma_0$ be the minimal $\sigma$-algebra of subsets of $\Omega$ including $\Sigma$. Then $s_i = \varphi(r_i)$, $i \in I$ is a minimal pre-Rademacher family in $\Sigma_0$. Using that $\Sigma_0$ is an algebra of subsets, we show that there is a unique (not necessarily positive) countably additive probability measure $\nu$ on $\Sigma_0$ satisfying
\begin{equation} \label{eq:mainprime}
\nu \Bigl( \bigcap_{k=1}^n \theta_k s_{i_k} \Bigr) = \frac{1}{2^n}
\end{equation}
for all finite collections of distinct indices $i_1, \ldots, i_n \in I$ and all collections of signs $\theta_1, \ldots, \theta_n \in \{-1,1\}$. Set $\mathcal{S} = (s_i)_{i \in I}$ and let $\nu'$ be the dyadic measure on the subalgebra $\Sigma_0(\mathcal{S})$ of $\Sigma_0$ generated by $\mathcal{S}$ which is finitely additive by Theorem~\ref{thm:dyadicfam}. Then \eqref{eq:mainprime} holds for $\nu'$ instead of $\nu$, and so our goal is to prove the existence and uniqueness of a countably additive extension $\nu$ of $\nu'$ to $\Sigma_0$.

First we prove that $\nu'$ is countably additive on $\Sigma_0(\mathcal{S})$. Indeed, let $(x_n)_{n=1}^\infty$ be a sequence in $\Sigma_0(\mathcal{S})$ with $x_{n+1} \subseteq x_n$ and $\bigcap_{n=1}^\infty x_n = \emptyset$. Say,
$$
x_n = \bigsqcup_{k=1}^{m_n} \bigcap_{j \in J_n} \theta_{n,k,j} s_j,
$$
where $m_n \in \mathbb N$, $J_n$ are finite subsets of $I$ and $\theta_{n,k,j} = \pm 1$ (see Lemma~\ref{le:tech1}). Using that $J_0 = \bigcup_{n=1}^\infty J_n$ is at most countable, we choose a countable $I_0 = \{i_n: \, n \in \mathbb N\} \subseteq I$ including $J_0$ with distinct $i_n$'s. Let $J_n'$ be the finite subset of $\mathbb N$ consisting of all $j \in \mathbb N$ such that $i_j \in J_n$. Then
\begin{equation} \label{eq:adjn}
x_n = \bigsqcup_{k=1}^{m_n} \bigcap_{j \in J_n'} \theta_{n,k,i_j} s_{i_j}.
\end{equation}
Let $(r_n)_{n=1}^\infty$ be the family of subsets of $[0,1)$ defined by \eqref{eqzxxxm} (by the obvious reason, with no loss of generality we may and do assume that $\mathbb N \cap I = \emptyset$, not to mix the elements of $(r_i)_{i \in I}$ and $(r_n)_{n=1}^\infty$). We set
$$
y_n = \bigsqcup_{k=1}^{m_n} \bigcap_{j \in J_n'} \theta_{n,k,i_j} r_j.
$$

By Lemma~\ref{le:tech2}, $y_{n+1} \subseteq y_n$ for all $n \in \mathbb N$. Now set $y = \bigcap_{n=1}^\infty y_n$ and $\Lambda = \bigl\{ (\lambda_n)_{n=1}^\infty: \, \lambda_n \in \{1, \ldots, m_n\} \bigr\}$. Given any $t \in y$ and $n \in \mathbb N$, by $k_n(t)$ we denote the integer $k \leq m_n$ such that $t \in \bigcap_{j \in J_n'} \theta_{n,k,i_j} r_j$. Then for any $\lambda = (\lambda_n)_{n=1}^\infty \in \Lambda$ we set $y_\lambda = \bigl\{ t \in y: \, (\forall n \in \mathbb N) \, k_n(t) = \lambda_n \bigr\}$ and observe that $y = \bigsqcup_{\lambda \in \Lambda} y_\lambda$. Let $\mu_0$ denote the Lebesgue measure on the Borel subsets of $[0,1)$. Since $\Lambda$ is countable, to show that $\mu_0(y) = 0$, it is enough to prove $\mu_0(y_\lambda) = 0$ for all $\lambda \in \Lambda$. By the definition of $y_\lambda$, one has
$$
y_\lambda = \bigcap_{n=1}^\infty \bigcap_{j \in J_n'} \theta_{n, \lambda_n, j} r_j.
$$
Taking another notation we get $y_\lambda = \bigcap_{j \in I_\lambda} \epsilon_j r_j$ for some subset $I_\lambda$ of $\mathbb N$ and signs $\epsilon_j = \pm 1$. Observe that $I_\lambda$ is infinite, because otherwise passing back to $x_n$'s we would obtain
$$
\emptyset \neq \bigcap_{j \in I_\lambda} \epsilon_j s_{i_j} \subseteq \bigcap_{n=1}^\infty x_n = \emptyset,
$$
a contradiction. Since $I_\lambda$ is infinite, $\mu_0(y_\lambda) = 0$ for all $\lambda \in \Lambda$ and, as a consequence, $\mu_0(y) = 0$. By the $\sigma$-additivity of the Lebesgue measure, $\lim_{n \to \infty} \mu_0(y_n) = \mu_0(y) = 0$.
Since
$$
\mu_0(y_n) = \sum_{k=1}^{m_n} \frac{1}{2^{|J_n'|}} = \nu'(x_n),
$$
we obtain $\lim_{n \to \infty} \nu'(x_n) = 0$. This is enough to show that $\nu'$ is countably additive on $\Sigma_0(\mathcal{S})$ by \cite[Proposition~1.3.3]{Bog}. By \cite[Proposition~1.3.10]{Bog}, there exists a unique countably additive extension $\nu: \Sigma_0 \to [0,1]$ of $\nu'$.

Now we define a measure $\mu: \mathcal{B}_\sigma(\mathcal{R}) \to [0,1]$ by setting $\mu(x) = \nu \bigl( \varphi(x) \bigr)$ for all $x \in \mathcal{B}_\sigma(\mathcal{R})$ and prove that $\mu$ has the desired properties. Observe that $\mu$ is well defined by the minimality of $\mathcal{R}$. For any $n \in \mathbb N$, distinct indices $i_1, \ldots, i_n \in I$ and signs $\theta_1, \ldots, \theta_n$ one has
$$
\mu \Bigl( \bigcap_{k=1}^n \theta_k r_{i_k} \Bigr) = \nu \left( \varphi \Bigl( \bigcap_{k=1}^n \theta_k r_{i_k} \Bigr) \right) = \nu \Bigl( \bigcap_{k=1}^n \theta_k \varphi(r_{i_k}) \Bigr) = \nu \Bigl( \bigcap_{k=1}^n \theta_k s_{i_k} \Bigr) = \frac1n
$$
by \eqref{eq:mainprime}, that is, \eqref{eq:main} holds. The countable additivity of $\nu$ on $\Sigma_0$ implies the countable additivity of $\mu$ on $\mathcal{B}_\sigma(\mathcal{R})$. Indeed, given any sequence $x_n \in \mathcal{B}_\sigma(\mathcal{R})$ with $x_{n+1} \leq x_n$ and $\inf_n x_n = \mathbf{0}$, one has that $\varphi(x_{n+1}) \subseteq \varphi(x_n)$ and $\inf_n \varphi(x_n) = \emptyset$ in $\Sigma_0$. Since $\Sigma_0$ is a $\sigma$-algebra, $\bigcap_{n=1}^\infty \varphi(x_n) \in \Sigma_0$, and so, $\bigcap_{n=1}^\infty \varphi(x_n) = \emptyset$ as a lower bound for the sequence $\bigl(\varphi(x_n)\bigr)_{n=1}^\infty$. By the countable additivity of $\nu$, $\lim_{n \to \infty} \nu \bigl(\varphi(x_n)\bigr) = 0$, that is, $\lim_{n \to \infty} \mu (x_n) = 0$. By \cite[Corollary~326G(a)]{F3}, $\mu$ is countably additive on $\mathcal{B}_\sigma(\mathcal{R})$.

It remains to show that $\mu$ is positive once $\mathcal{R}$ vanishes at infinity. Assume $z \in \mathcal{B}_\sigma(\mathcal{R})$ and $\mu(z) = 0$. Our goal is to show that $z = \mathbf{0}$. Choose a sequence $(z_n)_{n=1}^\infty$ in $\mathcal{B}(\mathcal{R})$ with $z_n \stackrel{\rm o}{\longrightarrow} z$. Let $(u_n)_{n=1}^\infty$ be a sequence in $\mathcal{B}_\sigma(\mathcal{R})$ such that $z_n \triangle z \leq u_n$ for all $n \in \mathbb N$ and $\bigcap_{n=1}^\infty u_n = \mathbf{0}$. Then
\begin{equation} \label{eq:zx.mh}
\forall n \in \mathbb N, \,\,\,\,\, z \leq z_n \cup (z \setminus z_n) \leq z_n \cup u_n.
\end{equation}
Say,
\begin{equation} \label{eq:zx.mh1}
z_n = \bigsqcup_{k=1}^{p_n} \bigcap_{j \in I_n} \vartheta_{n,k,j} r_j
\end{equation}
where $p_n \in \mathbb N$, $I_n$ are finite subsets of $I$ and $\vartheta_{n,k,j} = \pm 1$.
By \eqref{eq:zx.mh},
\begin{equation} \label{eq:zx.mh2}
\varphi(z) \subseteq \bigcap_{n=1}^\infty \bigl(\varphi(z_n) \cup \varphi(u_n)\bigr),
\end{equation}
where $\varphi$ is the isomorphism defined at the beginning of the proof, and by \eqref{eq:zx.mh1},
\begin{equation} \label{eq:zx.mh3}
\varphi(z_n) = \bigsqcup_{k=1}^{p_n} \bigcap_{j \in I_n} \vartheta_{n,k,j} \varphi(r_j)
\end{equation}

Since $u_{n+1} \leq u_n$ we have that $\varphi(z_{n+1}) \subseteq \varphi(z_n)$ for all $n \in \mathbb N$, and since $\bigcap_{n=1}^\infty u_n = \mathbf{0}$, we have $\bigcap_{n=1}^\infty \varphi(u_n) = \emptyset$. Hence for every $t \in \varphi(z)$ by \eqref{eq:zx.mh2}, there is $n_t \in \mathbb N$ such that $t \in \varphi(z_n)$ for all $n \geq n_t$, that is,
\begin{equation} \label{eq:zx.mh4}
\varphi(z) \subseteq \bigcup_{i=1}^\infty \bigcap_{n = i}^\infty \varphi(z_n).
\end{equation}

Given any $t \in \varphi(z_n)$, let $k_n(t)' \leq p_n$ be the least integer such that $t \in \bigcap_{j \in I_n} \vartheta_{n,k,j} \varphi(r_j)$ (such a number exists by \eqref{eq:zx.mh3}). Set $\Lambda' = \bigl\{ (\lambda_n)_{n=1}^\infty: \, \lambda_n \in \{1, \ldots, p_n\} \bigr\}$ and for every $i \in \mathbb N$ and $\lambda = (\lambda_n)_{n=1}^\infty \in \Lambda'$ set
$$
w_{i,\lambda} = \Bigl\{ t \in \bigcap_{n = i}^\infty \varphi(z_n): \,\,\, (\forall m \in \mathbb N) \, k_m'(t) = \lambda_m \Bigr\}.
$$

Observe that $|\Lambda'| = \aleph_0$ and for every $i \in \mathbb N$, $\bigcap_{n = i}^\infty \varphi(z_n) = \bigsqcup_{\lambda \in \Lambda'} w_{i,\lambda}$. On the other hand, for every $i \in \mathbb N$ and $\lambda = (\lambda_n)_{n=1}^\infty \in \Lambda'$
\begin{equation} \label{eq:zx.mh5}
w_{i,\lambda} = \bigcap_{n = i}^\infty \bigcap_{j \in I_n} \vartheta_{n,\lambda_n,j} \varphi(r_j) = \bigcap_{j \in M_{i,\lambda}} \varepsilon_{i,\lambda,j} \varphi(r_j) \subseteq \bigcap_{n = i}^\infty \varphi(z_n)
\end{equation}
for a suitable subset $M_{i,\lambda}$ of $I$ and signs $\varepsilon_{i,\lambda,j} = \pm 1$. We claim that $M_{i,\lambda}$ is infinite. Assume, on the contrary, that $|M_{i,\lambda}| = s$ for some $s \in \mathbb N$. Then, taking into account the last relation of \eqref{eq:zx.mh5} we deduce that
$$
\bigcap_{j \in M_{i,\lambda}} \varepsilon_{i,\lambda,j} r_j \leq \bigcap_{n = i}^\infty z_n
$$
and therefore
$$
2^{-s} = \mu \Bigl( \bigcap_{j \in M_{i,\lambda}} \varepsilon_{i,\lambda,j} r_j \Bigr) \leq \mu \Bigl( \bigcap_{n = i}^\infty z_n \Bigr) = 0,
$$
a contradiction. Thus, $M_{i,\lambda}$ is infinite. Since $\mathcal{R}$ vanishes at infinity, $\bigcap_{j \in M_{i,\lambda}} \varepsilon_{i,\lambda,j} r_j = \mathbf{0}$ and hence, $\bigcap_{j \in M_{i,\lambda}} \varepsilon_{i,\lambda,j} \varphi(r_j) = \emptyset$. By \eqref{eq:zx.mh5}, $w_{i,\lambda} = \emptyset$ for all $i \in \mathbb N$ and $\lambda = (\lambda_n)_{n=1}^\infty \in \Lambda'$. This yields that $\bigcap_{n = i}^\infty \varphi(z_n) = \emptyset$ for all $i \in \mathbb N$. By \eqref{eq:zx.mh4}, $\varphi(z) = \emptyset$ and hence $z = \mathbf{0}$.
\end{proof}

\begin{cor} \label{cor:mfhbbc}
Let $(r_i)_{i \in I}$ be a $\sigma$-complete Rademacher system in a Boolean algebra $\mathcal{B}$. Then there is a unique positive countably additive probability measure $\mu$ on $\mathcal{B}$ satisfying \eqref{eq:main}.
\end{cor}

Like in Definition~\ref{def:dyadic}, the unique and existed measure will be called the dyadic measure generated by the Rademacher system.

\begin{rem}
Let $\mathcal{R} = (r_i)_{i \in I}$ be a Rademacher family in a Boolean algebra $\mathcal{B}$. By Theorem~\ref{th:ccvmn740}, $\mathcal{B}_\sigma(\mathcal{R})$ is a measurable Boolean algebra, hence, $\mathcal{B}_\sigma(\mathcal{R})$ possesses the CCC and thus, $\mathcal{B}_\sigma(\mathcal{R}) = \mathcal{B}_\tau(\mathcal{R})$. So, the words ``$\sigma$-complete Rademacher family'' we replace with ``complete Rademacher family''.
\end{rem}

\begin{defn} \label{def:dyadic}
Let $\mathcal{R} = (r_i)_{i \in I}$ be a complete Rademacher family in a Boolean algebra $\mathcal{B}$. The countably additive positive probability measure $\mu: \mathcal{B} \to [0,1]$ satisfying \eqref{eq:main} is called the \emph{dyadic measure} on $\mathcal{B}$ generated by $\mathcal{R}$.
\end{defn}

\section{A characterization of homogeneous measurable algebras} \label{sec:charhma}

\begin{thm} \label{th:djjfh75mach}
Dedekind $\sigma$-complete Boolean algebras admitting complete Rademacher systems of the same cardinality are isomorphic.
\end{thm}

Before the proof, we recall some more information on the order convergence. Limit inferior and limit superior of an order bounded sequence $(x_n)_{n=1}$ in a Dedekind $\sigma$-complete lattice $E$ are defined by $\liminf_n x_n = \sup_n \inf_{k \geq n} x_k$ and $\limsup_n x_n = \inf_n \sup_{k \geq n} x_k$ respectively. Anyway, $\liminf_n x_n \leq \limsup_n x_n$. If $x = \liminf_n x_n = \limsup_n x_n$ then it is said the sequence $(x_n)_{n=1}$ order converges to $x$ and written $x_n \overset{\rm o}\longrightarrow x$. This new notion of the order convergence in a Boolean algebra coincides with the one given in the introduction.

\begin{proof}
Let $\mathcal{B}'$ and $\mathcal{B}''$ be Dedekind $\sigma$-complete Boolean algebras with complete Rademacher systems $\mathcal{R}' = (r_i')_{i \in I}$ and $\mathcal{R}'' = (r_i'')_{i \in I}$ respectively. We construct an isomorphism $\varphi: \mathcal{B}' \to \mathcal{B}''$ as follows. First we define $\varphi$ on $\mathcal{B}'(\mathcal{R}')$ by setting
$$
\varphi \Bigl( \bigsqcup_{k=1}^m \bigcap_{j \in J} \theta_{k,j} r_j' \Bigr) = \bigsqcup_{k=1}^m \bigcap_{j \in J} \theta_{k,j} r_j''
$$
for every $m \in \mathbb N$, every finite subset $J$ of $I$ and each signs $\theta_{n,k,j} = \pm 1$.

Take any $x' \in \mathcal{B}'$. Choose a sequence $(x_n')_{n=1}^\infty$ in $\mathcal{B}'(\mathcal{R}')$ with $x_n' \overset{\rm o}\longrightarrow x'$. Then $x' = \liminf_n x_n' = \limsup_n x_n'$. Say,
$$
x_n' = \bigsqcup_{k=1}^{m_n} \bigcap_{j \in J_n} \theta_{n,k,j} r_j',
$$
where $m_n \in \mathbb N$, $J_n$ are finite subsets of $I$ and $\theta_{n,k,j} = \pm 1$ (see Lemma~\ref{le:tech1}). Our goal is to prove the existence of the order limit of the sequence
$$
x_n'' = \varphi(x_n') = \bigsqcup_{k=1}^{m_n} \bigcap_{j \in J_n} \theta_{n,k,j} r_j''
$$
in $\mathcal{B}''$. Let $\mu'$ and $\mu''$ be the dyadic measures on $\mathcal{B}'$ and $\mathcal{B}''$ generated by $\mathcal{R}'$ and $\mathcal{R}''$ respectively. By the definitions, $\mu'(x_n') = \mu''(x_n'')$ for all $n \in \mathbb N$. Set $\underline{x}'' = \liminf_n x_n''$ and $\overline{x}'' = \limsup_n x_n''$. Then, using the countable additivity of $\mu'$ and $\mu''$ we obtain
\begin{align*}
\mu''(\underline{x}'') &= \liminf_n \mu''(x_n'') = \liminf_n \mu'(x_n') = \mu'(x')\\
&= \limsup_n \mu'(x_n') = \limsup_n \mu''(x_n'') = \mu''(\overline{x}'').
\end{align*}

Taking into account that $\underline{x}'' \leq \overline{x}''$, one has
$$
\mu''(\overline{x}'' \setminus \underline{x}'') = \mu''(\overline{x}'') - \mu''(\underline{x}'') = 0.
$$
By positivity of $\mu''$, $\underline{x}'' = \overline{x}''$ which yields the existence of the order limit $x'' = \underline{x}'' = \overline{x}''$ of the sequence $(x_n'')_{n=1}^\infty$. We set $\varphi(x') = x''$. The proof of the independence of $x''$ on the choice of a sequence $(x_n')_{n=1}^\infty$ and that $\varphi$ is a bijection is standard. By the above, $\varphi: \mathcal{B}' \to \mathcal{B}''$ is an order continuous and measure preserving map, that is, $\mu'' \bigl( \varphi(x) \bigr) = \mu'(x)$ for all $x \in \mathcal{B}'$.

It remains to show that $\varphi$ is order preserving. Assume $x,y \in \mathcal{B}'$ with $x \leq y$. Choose any sequences $(x_n)_{n=1}^\infty$ and $(y_n)_{n=1}^\infty$ in $\mathcal{B}'(\mathcal{R}')$ with $x_n \overset{\rm o}\longrightarrow x$ and $y_n \overset{\rm o}\longrightarrow y$. Set $x_n' = x_n \cap y_n$ and $y_n' = x_n \cup y_n$ for all $n \in \mathbb N$. Then $x_n', y_n' \in \mathcal{B}'(\mathcal{R}')$, $x_n' \overset{\rm o}\longrightarrow x \cap y = x$ and $y_n' \overset{\rm o}\longrightarrow x \cup y = y$. Since $x_n' \leq y_n'$, by Lemma~\ref{le:tech2}, $\varphi(x_n') \leq \varphi(y_n')$ for all $n$. By the order continuity of $\varphi$, $\varphi(x) \leq \varphi(y)$.
\end{proof}

As a consequence of the Maharam theorem~\ref{th:Mah1} and Theorem~\ref{th:djjfh75mach} we obtain the following new characterization of homogeneous measurable algebras.

\begin{thm} \label{th:mach}
A Dedekind $\sigma$-complete Boolean algebra $\mathcal{B}$ is homogeneous measurable if and only if there is a complete Rademacher family in $\mathcal{B}$.
\end{thm}

\section{Rademacher and Haar systems in Riesz spaces} \label{secfive}

$\,\,\,\,\,\,$\textbf{Necessary information on Riesz spaces.} \label{subsfragm} Recall some well known definitions. Let $E$ be a Riesz space. Two elements $x,y \in E$ are called \textit{disjoint} (or orthogonal) if $|x| \wedge |y| = 0$ and this fact is written as $x \bot \, y$. Two subsets $A,B \subseteq E$ are \textit{disjoint} if $x \bot \, y$ for all $x \in A$ and $y \in B$. For any subset $A \subseteq E$ by $A^d$ we denote the set $A^d =  \{x \in E: \,\, A \,\, \mbox{and} \,\, \{x\} \,\, \mbox{are disjoint}  \}$. The notation $x = \bigsqcup_{k=1}^n x_k$ means that $x = \sum_{k=1}^n x_k$ and $x_i \bot \, x_j$ if $i \neq j$.

An element $x$ of a Riesz space $E$ is called a \textit{fragment} (in other terminology, a component) of an element $y \in E$, provided $x \bot (y-x)$. The notation $x \sqsubseteq y$ means that $x$ is a fragment of $y$. Given any $e \in E^+$, the set $\mathcal C_e^X$ of all fragments of $e$ is a Boolean algebra with respect to the lattice operations $\vee$ and $\wedge$, consisting precisely of all extreme points of the order segment $[0,e]$ \cite[p.~37]{ABu}.

A \textit{sublattice} of a Riesz space $E$ is a vector subspace $F$ of $E$ which is a lattice itself with respect to the same ordering. A subset $A$ of a Riesz space $E$ is called \textit{solid} if for any $x \in A$ and $y \in E$ the condition $|y| \leq |x|$ implies that $y \in A$. A solid vector subspace is called an \textit{ideal}. An order closed ideal is called  a \textit{band}. A band $I$ of a Riesz space $E$ is called a \textit{projection band} if $E = I \oplus I^d$.

By $A_e$ and $B_e$ we denote the ideal and the band in $E$ generated by an element $e \in E$.

Let $F$ be an ideal of a Banach lattice $E$. An element $e \in F^+$ is called a \textit{weak unit} of $F$ provided for every $x \in F$ the condition $x \bot e = 0$ implies that $x = 0$.

Recall that an element $e$ of a Riesz space $E$ is called a \textit{projection element} if $B_e$ is a projection band. A Riesz space $E$ is said to have the \textit{principal projection property} if every element of $E$ is a projection element.

Let $e \geq 0$ be a projection of an element of $E$. Then for every $x \in X$ by $P_e (x)$ we denote the projection of $x$ to $B_e$. By \cite[Theorem~3.13]{ABu}, $P_e (x) = \bigvee_{n=1}^\infty (x \wedge ne)$ for every $x \in E^+$. One can show that $P_e (x) \sqsubseteq x$ for all $x$ and $e$.

A Riesz space $E$ is called \emph{Dedekind complete} (\emph{Dedekind} $\sigma$-\emph{complete}) if every order bounded from above nonempty subset of (sequence in) $E$ has the least upper bound in $E$. Every Dedekind $\sigma$-complete Riesz space has the principal projection property \cite[p.~36]{ABu}.

A decreasing (in the non-strict sense) net $(x_\alpha)$ in a Riesz space $E$ is said to be \textit{order converging} to an element $x \in E$ (notation $x_\alpha \downarrow x$) if $\inf_{\alpha} x_\alpha = x$. Likewise, the notation $x_\alpha \uparrow x$ means that $(x_\alpha)$ is an increasing net with $\sup_{\alpha} x_\alpha = x$. More generally, it is said that a net $(x_\alpha)_{\alpha \in \Lambda}$ in $E$ \textit{order converges} to an element $x \in E$ (notation $x_\alpha \stackrel{\rm o}{\longrightarrow} x$ if there exists a net $(u_\alpha)_{\alpha \in \Lambda}$ in $E$ with the same index set such that $u_\alpha \downarrow 0$ and $|x_\beta - x| \leq u_\beta$ for all  $\beta\in \Lambda$.

Let $E$ be a Riesz space, $0 < e \in E$. Any element $x \in E$ of the form $x = \bigsqcup_{k=1}^m a_k x_k$, where $m \in \mathbb N$, $a_k \in \mathbb R$ and $x_k \in \mathfrak{F}_e$ is called an $e$-step function. The following well known theorem is very essential for the sequel, see
\cite[Theorem~6.8]{ABu}.

\begin{thm}[Freudenthal's Spectral Theorem] \label{th:Frsssth}
Let $E$ be a Riesz space with the principal projection property and $0 < e \in E$. Then for every $x \in A_e$ there exists a sequence $(u_n)_{n=1}^\infty$ of $e$-step functions such that
$$
\forall n, \,\, 0 \leq x - u_n \leq \frac1n \, e \,\,\,\,\, \mbox{and} \,\,\,\,\, u_n \uparrow x.
$$
\end{thm}

In what follows, we consider Archimedean Riesz spaces only (see \cite{ABu} for more details).

$\,\,\,\,\,\,$\textbf{Signed fragments, pre-Rademacher and Rademacher systems.} \begin{defn}
Let $E$ be a Riesz space and $0 \neq e \in E$. An element $x \in E$ is called a \emph{signed fragment} of $e$ provided that $|x| \sqsubseteq |e|$.
\end{defn}

By $|\mathfrak{F}|_e$ we denote the set of all signed fragments of $e$. It is clear from the definition that $|\mathfrak{F}|_e = |\mathfrak{F}|_{|e|}$, so we assume $e > 0$ in what follows. For convenience of notation, we introduce the multiplication operation on $|\mathfrak{F}|_e$ to be the following disjoint sum:
$$
x \cdot y = x^+ \wedge y^+ + x^- \wedge y^- - x^+ \wedge y^- - x^- \wedge y^+
$$
for all $x,y \in |\mathfrak{F}|_e$. To prove the following proposition is a standard technical exercise.

\begin{prop} \label{pr:sysysysysd}
Let $E$ be a Riesz space and $0 < e \in E$.
\begin{enumerate}
  \item The binary operation $\cdot$ is well defined on $|\mathfrak{F}|_e$, that is, $x \cdot y \in |\mathfrak{F}|_e$ for all $x,y \in |\mathfrak{F}|_e$.
  \item $|\mathfrak{F}|_e$ is an Abelian group relative to the operation $\cdot$, unit $e$, and the inverse element $x^{-1} = x$ for all $x \in |\mathfrak{F}|_e$.
\end{enumerate}
\end{prop}

\begin{defn}
Let $E$ be a Riesz space and $0 < e \in E$. Let $\mathfrak{F}_e$ denote the Boolean algebra of fragments of $e$ and let $\widehat{\mathcal{R}} = (\widehat{r}_i)_{i \in I}$ be a pre-Rademacher family in $\mathfrak{F}_e$. The following system of elements of $|\mathfrak{F}|_e$
\begin{equation} \label{eq:defrads}
\mathcal{R} = (r_i)_{i \in I}, \,\,\,\,\, \mbox{wrere} \,\,\, r_i = 2 \widehat{r}_i - e, \,\, i \in I
\end{equation}
is called a \emph{pre-Rademacher system} in $E$ (more precisely, a \emph{pre-Rademacher system on} $e$).
\end{defn}

Analogously we define \emph{minimal, vanishing at infinity, $\sigma$-complete pre-Rademacher system} and a \emph{Rademacher system} on an element $e > 0$.

Observe that
\begin{equation} \label{eq:defrads2}
\widehat{r}_i = \frac{e + r_i}{2}, \,\,\, e - \widehat{r}_i = \frac{e - r_i}{2},
\end{equation}
and $|r_i| = e$ for all $i \in I$.

\begin{thm} \label{th:ssssmdnd750}
Let $E$ be a Dedekind complete Riesz space. Then for any element $e > 0$ of $E$ and any infinite cardinal $\aleph_\alpha$ the following assertions are equivalent.
\begin{enumerate}
  \item There exists a complete Rademacher system $(r_i)_{i \in I}$ on $e$ of cardinality $|I| = \aleph_\alpha$.
  \item The ideal $A_e$ is Riesz isomorphic to $L_\infty(D^{\omega_\alpha}, \Sigma_{\omega_\alpha}, \mu_{\omega_\alpha})$. Moreover, there exists a Riesz isomorphism $T: A_e \to L_\infty(D^{\omega_\alpha}, \Sigma_{\omega_\alpha}, \mu_{\omega_\alpha})$ such that the restriction $T|_{\mathfrak{F}_e}: \mathfrak{F}_e \to \widehat{\Sigma}_{\omega_\alpha}$ is a Boolean isomorphism.
\end{enumerate}

In the case where $\alpha = 0$ one can replace $L_\infty(D^{\omega_0}, \Sigma_{\omega_0}, \mu_{\omega_0})$ with $L_\infty[0,1]$ and $\widehat{\Sigma}_{\omega_0}$ with $\widehat{\Sigma}$ in (2).
\end{thm}

\begin{proof}
(1) $\Rightarrow$ (2). By Theorem~\ref{th:djjfh75mach}, there is a Boolean isomorphism $\varphi: \mathfrak{F}_e \to \widehat{\Sigma}_{\omega_\alpha}$. Define a linear operator $T_0: G \to L_0(D^{\omega_\alpha}, \Sigma_{\omega_\alpha}, \mu_{\omega_\alpha})$, where $G$ is the Riesz subspace of $E$ of all $e$-step functions and $L_0(D^{\omega_\alpha}, \Sigma_{\omega_\alpha}, \mu_{\omega_\alpha})$ is the Riesz space of all equivalence classes of measurable functions from $D^{\omega_\alpha}$ to $\mathbb R$, by setting
$$
T_0 \Bigl( \sum_{k=1}^m a_k x_k \Bigr) = \sum_{k=1}^m a_k \varphi(x_k)
$$
for every $\sum_{k=1}^m a_k x_k \in G$. Observe that, $T_0$ is an injective lattice homomorphism, and that $T_0(G)$ coincides with the Riesz subspace of $L_0(D^{\omega_\alpha}, \Sigma_{\omega_\alpha}, \mu_{\omega_\alpha})$ of all $\mathbf{1}$-step functions, where $\mathbf{1}$ is the equivalence class containing the constant function $1$. We show that $T_0$ is order continuous. Indeed, let $(u_\lambda)$ be a net in $G$ with $u_\lambda \downarrow 0$. Our goal is to show that $\inf_\lambda T_0(u_\lambda) = 0$. If this were not true, there would exist $z \in L_0(D^{\omega_\alpha}, \Sigma_{\omega_\alpha}, \mu_{\omega_\alpha})$ with $0 < z \leq T_0(u_\lambda)$ for all $\lambda$. Obviously, the set of all simple functions is order dense in $L_0(D^{\omega_\alpha}, \Sigma_{\omega_\alpha}, \mu_{\omega_\alpha})$, and so, by the above, there is an element $y \in G$ with $0 < T_0(y) \leq z \leq T_0(u_\lambda)$ for all $\lambda$. Hence, $0 < y \leq u_\lambda$ for all $\lambda$, which contradicts the assumption $u_\lambda \downarrow 0$. By \cite[Theorem~4.3]{ABu}, $T_0$ is order continuous. Observe that, by Freudenthal's Spectral Theorem, $G$ is an order dense Riesz subspace of $A_e$. Since $L_0(D^{\omega_\alpha}, \Sigma_{\omega_\alpha}, \mu_{\omega_\alpha})$ is a laterally complete Riesz space\footnote{that is, every disjoint set of positive elements has a supremum}, by \cite[Theorem~7.20]{ABu}, the formula
\begin{equation} \label{eq:sdnbb}
T(x) = \sup \bigl\{ T(y): \,\, (y \in G) \, \& \, (0 \leq y \leq x) \bigr\}
\end{equation}
defines an extension $T: A_e \to L_0(D^{\omega_\alpha}, \Sigma_{\omega_\alpha}, \mu_{\omega_\alpha})$ which is an order continuous lattice homomorphism. By the order density of the set of all simple functions in $L_0(D^{\omega_\alpha}, \Sigma_{\omega_\alpha}, \mu_{\omega_\alpha})$ and \eqref{eq:sdnbb}, $T$ is injective. It remains to observe that $T(A_e) = L_\infty(D^{\omega_\alpha}, \Sigma_{\omega_\alpha}, \mu_{\omega_\alpha})$.

The implication (2) $\Rightarrow$ (1) is obvious.
\end{proof}

$\,\,\,\,\,\,$\textbf{Haar and other systems.} Let $E$ be a Riesz space, $0 < e \in E$ and $\mathcal{R} = (r_n)_{n=1}^\infty$ a Rademacher system on $e$. We define \emph{the Haar system generated by} $\mathcal{R}$ as follows. For every $k \in \mathbb N$ let $(\epsilon_{k,j})_{j=1}^\infty$ be the sequence of signs $\epsilon_{k,j} = \pm 1$ defined as follows. Let $k -1 = \sum_{j=1}^\infty a_{k,j} 2^j$ be the dyadic expansion with digits $a_{k,j} \in \{0,1\}$. Then set $\epsilon_{k,j} = 1 - 2 a_{k,j}$ for all $k,j \in \mathbb N$. Observe that for every fixed $n = 2,3, \ldots$ the sequences $(\epsilon_{k,j})_{j=1}^{n-1}$ for $k = 1, \ldots, 2^{n-1}$ are distinct and all possible sequences of signs of length $n-1$. Then we define $h_1 = e$, $h_2 = r_1$ and for every $n = 2,3, \ldots$ and every $k = 1, \ldots, 2^n$
$$
h_{2^n+k} = r_n \cdot \bigwedge_{j=1}^{n-1} a_{k,j} \frac{e + \epsilon_{k,j} r_j}{2}
$$
to be the \emph{Haar system} $(h_i)_{i=1}^\infty$ on $e$. We say that a Haar system is \emph{complete} if the corresponding Rademacher system is complete.

Let $F$ be a linear subspace of a Riesz space $E$. Following \cite{GKP}, a sequence $(x_n)_{n=1}^\infty$ in $F$ is called an \textit{order Schauder basis} of $F$ if for every $x \in F$ there exists a unique sequence of scalars $(a_n)_{n=1}^\infty$ such that
$$
\sum\limits_{k=1}^n a_k x_k \stackrel{\rm o}{\longrightarrow} x.
$$
A sequence $(x_n)_{n=1}^\infty$ in a Riesz space $E$ is called an \textit{order Schauder basic sequence} if it is an order Schauder basis of the sequential order closure of the linear span of $(x_n)_{n=1}^\infty$.

\begin{thm}
Let $(h_n)_{n=1}^\infty$ be a complete Haar system on an element $e$ of a Dedekind complete Riesz space $E$. Then $(h_n)_{n=1}^\infty$ is an order Schauder basis of the ideal $A_e$ generated by $e$.
\end{thm}

\begin{proof}
By Theorem~\ref{th:ssssmdnd750}, there exists a Riesz isomorphism $T: A_e \to L_\infty[0,1]$ such that the restriction $T|_{\mathfrak{F}_e}: \mathfrak{F}_e \to \widehat{\Sigma}$ is a Boolean isomorphism which sends $(h_n)_{n=1}^\infty$ to the usual Haar system $(\widetilde{h}_n)_{n=1}^\infty$ on $[0,1]$. By \cite{GKP}, $(\widetilde{h}_n)_{n=1}^\infty$ is an order Schauder basis in $L_\infty[0,1]$, and hence, $(h_n)_{n=1}^\infty$ is an order Schauder basis of the ideal $A_e$.
\end{proof}

Note that a Haar system in a Riesz space need not be an order Schauder basic sequence. Indeed, the usual Haar system $(h_n)_{n=1}^\infty$ is not an order Schauder basis of $L_1[0,1]$ \cite{GKP}. Actually, $(h_n)_{n=1}^\infty$ is an order Schauder basis of the ideal $A_\mathbf{1}$ in $L_1[0,1]$ generated by the constant unit function $\mathbf{1}$, which coincides with $L_\infty[0,1]$. However, since the sequential order closure of $(h_n)_{n=1}^\infty$ in $L_1[0,1]$ is $L_1[0,1]$ itself, $(h_n)_{n=1}^\infty$ is not an order Schauder basic sequence.

Recall that an element $e > 0$ of a Riesz space $E$ is called an \emph{order unit} of $E$ if for every $x \in E$ there is a scalar $a > 0$ such that $|x| \leq a e$. In other words, $e$ is an order unit of $E$ provided $A_e = E$.

\begin{cor}
Let $E$ be a Riesz space with an order unit $e > 0$. Then every complete Haar system on $e$ is an order Schauder basis of $E$.
\end{cor}

\begin{prob} \label{pr}
Let $E$ be a Riesz space, $0 < e \in E$ and let $(h_n)_{n=1}^\infty$ be a complete Haar system on $e$. Characterize those Riesz subspaces $X$ of $B_e$ such that $(h_n)_{n=1}^\infty$ is an order Schauder basis of $E$.
\end{prob}

To confirm that Problem~\ref{pr} is too involved and interesting, note that for the case where $E = L_0[0,1]$, $e = \mathbf{1}$ and $(h_n)_{n=1}^\infty$ is the usual Haar system on $[0,1]$, a partial answer one can find in \cite{GKP}. More precisely, it is known that $(h_n)_{n=1}^\infty$ is an order Schauder basis for $L_p[0,1]$ with $1 < p \leq \infty$, and is not for $L_1[0,1]$, and the proofs of these results uses further results on Fourier series with respect to the Haar system. On the other hand, it is unknown of whether the Riesz space $L_1[0,1]$ has an order Schauder basis \cite{GKP}.

Using a complete Rademacher system on an element $e > 0$ of a Riesz space $E$, one can define some other classical systems, like Walsh \cite{PR} and Olevsky systems \cite{NS}.

\section{Integration in Riesz spaces with respect to a Rademacher system} \label{sec:int}

By an infinite disjoint sum $\bigsqcup_{n=1}^\infty x_n$ of elements $x_n$ of a Riesz space $E$ we understand the difference
$$
\bigsqcup_{n=1}^\infty x_n = \sup_n x_n^+ - \sup_n x_n^-
$$
if both suprema exist and $x_i \perp x_j$ for all $i \neq j$.

\begin{defn}
Let $E$ be a Riesz space, $0 < e \in E$. By an $e$-\emph{simple element} we understand any element $x \in E$ admitting an expansion
$$
x = \bigsqcup_{n=1}^\infty a_n x_n,
$$
where $a_n \in \mathbb R$, $x_n \in \mathfrak{F}_e$ for all $n \in \mathbb N$.
\end{defn}

Obviously, the set $\mathcal{S}_e$ of all $e$-simple elements is a sublattice of $E$.

Let $E$ be a Riesz space, $0 < e \in E$, $\widehat{\mathcal{R}} = (\widehat{r}_i)_{i \in I}$ be a complete Rademacher family in $\mathfrak{F}_e$ and $\mathcal{R} = (r_i)_{i \in I}$ be the Rademacher system on $e$ defined by \eqref{eq:defrads}. Let $\mu$ be the dyadic probability positive measure on $\mathfrak{F}_e$ generated by $\mathcal{R}$.

\begin{defn}
An $e$-simple element $x = \bigsqcup_{n=1}^\infty a_n x_n$, where $a_n \in \mathbb R$, $x_n \in \mathfrak{F}_e$ for all $n \in \mathbb N$, is called $\mathcal{R}$-\emph{integrable} if $\sum_{n=1}^\infty |a_n| \mu (x_n) < \infty$. In this case we define the integral of $x$ over $e$ with respect to $\mathcal{R}$ by setting
$$
\int_e x \, d \mu = \sum_{n=1}^\infty a_n \mu (x_n).
$$
\end{defn}

One can show that, if $x$ is an $\mathcal{R}$-integrable $e$-simple element then the above definition of the integral of $x$ does not depend on an expansion of $x$ as an infinite disjoint sum of scalar multiples by fragments of $e$.

The set of all $\mathcal{R}$-integrable $e$-simple elements we denote by $\mathcal{S}_e^1$. Using standard tools, one can prove the following properties of the above notion.

\begin{prop} \label{pr:masl}
\begin{enumerate}
  \item If $x \in \mathfrak{F}_e$ and $a \in \mathbb R$ then $ax \in \mathcal{S}_e^1$ and $\displaystyle{\int_e ax \, d \mu = a \mu(x)}$.
  \item If $x,y \in \mathcal{S}_e^1$ and $a,b \in \mathbb R$ then $ax+by \in \mathcal{S}_e^1$ and $$\int_e (ax+by) \, d \mu = a \int_e x \, d \mu + b \int_e y \, d \mu.$$
  \item If $x,y \in \mathcal{S}_e^1$ and $x \leq y$ then $\displaystyle{\int_e x \, d \mu \leq \int_e y \, d \mu}$.
  \item If $x \in \mathcal{S}_e$ and $|x| \leq \lambda e$ for some $\lambda \geq 0$ then $x \in \mathcal{S}_e^1$ and $\displaystyle{\Bigl| \int_e x \, d \mu \Bigr| \leq \lambda.}$
\end{enumerate}
\end{prop}

The next important for the integration statement is a version of Freudenthal's Spectral Theorem~\ref{th:Frsssth}, which follows from Freudenthal's Spectral Theorem.

\begin{thm} \label{th:Freudlike}
Let $E$ be a $\sigma$-Dedekind complete Riesz space and $0 < e \in E$. Then for every $x \in B_e$ there exists a sequence $(u_n)_{n=1}^\infty$ of $e$-simple elements such that
\begin{equation} \label{eq:dddggs}
\forall n, \,\, 0 \leq x - u_n \leq \frac1n \, e \,\,\,\,\, \mbox{and} \,\,\,\,\, u_n \uparrow x.
\end{equation}
\end{thm}

Remark that the important difference from Theorem~\ref{th:Frsssth} is that one can approximate any element from the band generated by $e$ by a sequence of $e$-simple elements.

\begin{proof}[Sketch of proof]
With no loss of generality, we may and do assume that $x \geq 0$. We define a sequence $(x_m)_{m=1}^\infty$ by setting
$$
x_1 = P_{(e-x)^+} (x) \,\,\,\,\, \mbox{and} \,\,\,\,\, x_{m+1} = P_{((m+1)e-x)^+} (x) - P_{(me-x)^+} (x) \,\,\,\,\, \mbox{for} \,\,\,\,\, m = 1,2, \ldots.
$$

Observe that
\begin{equation} \label{eq:dddggs0}
x = \bigsqcup_{m=1}^\infty x_m.
\end{equation}

For every $m \in \mathbb N$, we set $e_m = P_{x_m} e$. Then $(e_m)_{m=1}^\infty$ is a disjoint sequence of fragments of $e$ and $x_m \in A_{e_m}$ for all $m \in \mathbb N$. Using Theorem~\ref{th:Frsssth}, for every $m \in \mathbb N$ we choose a sequence $(u_{m,n})_{n=1}^\infty$ of $e_m$-step functions with
\begin{equation} \label{eq:dddggs1}
\forall n, \,\, 0 \leq x_m - u_{m,n} \leq \frac1n \, e_m \,\,\,\,\, \mbox{and} \,\,\,\,\, u_{m,n} \uparrow x_m \,\,\, \mbox{as} \,\, n \to \infty.
\end{equation}

Since $x - e \leq u_{m,n} \leq x$ for all $n,m \in \mathbb N$ and $E$ is Dedekind $\sigma$-complete, for every $n \in \mathbb N$ there exists the following infinite disjoint sum
$$
u_n = \bigsqcup_{m=1}^\infty u_{m,n}.
$$
By \eqref{eq:dddggs0} and \eqref{eq:dddggs1}, the sequence $(u_n)_{n=1}^\infty$ has the desired properties.
\end{proof}

From now on, assume in addition, that $E$ is $\sigma$-Dedekind complete.

\begin{defn}
An element $x \in B_e$ is called $\mathcal{R}$-\emph{integrable} if there exists a sequence $(u_n)_{n=1}^\infty$ of integrable $e$-simple elements possessing \eqref{eq:dddggs}.
\end{defn}

\begin{prop} \label{pr:ksjdhwn1}
Let $x \in B_e$ be $\mathcal{R}$-integrable and let $(u_n)_{n=1}^\infty$ be a sequence of integrable $e$-simple elements possessing \eqref{eq:dddggs}. Then there exists the limit
\begin{equation} \label{eq:dddggs2}
\lim_{n \to \infty} \int_e u_n \, d \mu.
\end{equation}
\end{prop}

\begin{proof}
We show that $\displaystyle{\int_e u_n \, d \mu}$, $n = 1,2, \ldots$ is a Cauchy sequence. Indeed, if $n < m$ then
$$
|u_n - u_m| = u_n - u_m = (x - u_m) - (x - u_n) \leq \frac1m \, e - 0 \leq \frac1n \, e.
$$

Hence, by (4) of Proposition~\ref{pr:masl},
$$
\Bigl| \int_e u_n \, d \mu - \int_e u_m \, d \mu \Bigr| = \Bigl| \int_e (u_n - u_m) \, d \mu \Bigr| \leq \frac1n \, .
$$
\end{proof}

The following proposition is proved by the well known scheme.

\begin{prop} \label{pr:ksjdhwn2}
Let $x \in B_e$ be $\mathcal{R}$-integrable. Then the limit \eqref{eq:dddggs2} does not depend on a sequence $(u_n)_{n=1}^\infty$ of integrable $e$-simple elements possessing \eqref{eq:dddggs}.
\end{prop}

\begin{defn}
Let $x \in B_e$ be $\mathcal{R}$-integrable. The limit \eqref{eq:dddggs2}, where $(u_n)_{n=1}^\infty$ is a sequence of integrable $e$-simple elements possessing \eqref{eq:dddggs}, is called the \emph{integral of} $x$ \emph{over} $e$ \emph{with respect to} $\mathcal{R}$ and denoted by
$$
\int_e x \, d\mu = \lim_{n \to \infty} \int_e u_n \, d \mu.
$$
\end{defn}

One can prove that the defined above integral possesses all the usual properties of the Lebesgue integral. In particular, the integrability of $x$ and $|x|$ are equivalent and one has
$$
\Bigl| \int_e x \, d \mu \Bigr| \leq \int_e |x| \, d \mu.
$$

The set of all $\mathcal{R}$-integrable functions we denote by $L_1(\mathcal{R})$. It is immediate that $L_1(\mathcal{R})$ is an ideal of $E$ and a normed lattice with respect to the norm
$$
\|x\| = \int_e |x| \, d \mu.
$$

However, $L_1(\mathcal{R})$ need not be a Banach space, as the following simple example shows. Let $E = L_\infty[0,1]$ and $e = \mathbf{1}$, the constant function equals $1$ a.e. If $\mathcal{R}$ is the usual Rademacher system on $[0,1)$ then $L_1(\mathcal{R}) = L_\infty[0,1]$, which is not a Banach space with respect to the $L_1$-norm.

Another example shows that $L_1(\mathcal{R})$ need not be a band in $E$. Indeed, let $E = L_0[0,1]$, $e = \mathbf{1}$ and let $\mathcal{R}$ be the usual Rademacher system on $[0,1)$. Then $L_1(\mathcal{R}) = L_1[0,1]$, which is not a band in $L_0[0,1]$ by the obvious reason.

Similarly, one can define the normed spaces $L_p(\mathcal{R})$ with $1 < p < \infty$, as well as Lorentz and Orlicz spaces using the integral with respect to $\mathcal{R}$.

\section{A partial answer to Problem~\ref{pr:main}} \label{prmainpartans}

The following theorem which is a consequence of the above results gives a partial answer to Problem~\ref{pr:main} in its part concerning Rademacher systems.

\begin{thm} \label{th:partans}
Let $(r_i)_{i \in I}$ be a system in a Riesz space $E$. Then the following assertions are equivalent.
\begin{enumerate}
  \item There exist a probability space $(\Omega, \Sigma, \mu)$, a Riesz subspace $F$ of $E$ containing $(r_i)_{i \in I}$ and a Riesz into isomorphism $T: F \to L_0(\mu)$ such that $(Tr_i)_{i \in I}$ is a system of independent $\pm 1$-valued random variables with the distribution $\mu\{Tr_i = \pm 1\} = 1/2$, $i \in I$;
  \item $(r_i)_{i \in I}$ possesses the following properties:
  \begin{enumerate}
    \item there exists $e \in E$, $e > 0$ such that $(\forall i \in I) \,\, |r_i| = e$;
    \item for every finite $J \subseteq I$ and every signs $\theta_j \in \{+, -\}$, $j \in J$ one has $\bigwedge_{j \in J} r_j^{\theta_j} > 0$;
    \item for every infinite $J \subseteq I$ and every signs $\theta_j \in \{+, -\}$, $j \in J$ one has $\bigwedge_{j \in J} r_j^{\theta_j} = 0$;
    \item for every $i_0 \in I$, $r_{i_0}^+$ does not belong to the smallest order closed subalgebra of $\mathfrak{F}_e$ containing all $r_i^+$ with $i \in I \setminus \{i_0\}$.
  \end{enumerate}
\end{enumerate}
\end{thm}

\begin{proof}
The implication (1) $\Rightarrow$ (2) is obvious (properties (a)-(d) are clear for a system of independent $\pm 1$-valued random variables with the distribution $\mu\{Tr_i = \pm 1\} = 1/2$, $i \in I$, and these properties are preserved by a Riesz isomorphism.

(2) $\Rightarrow$ (1). By (2), $\mathcal{R} = (r_i^+)_{i \in I}$ is a Rademacher system on $e$ in the Boolean algebra $\mathfrak{F}_e$. By Theorem~\ref{th:ccvmn740}, there exists a unique positive countably additive probability measure $\mu'$ (called the dyadic measure) on the smallest $\sigma$-complete subalgebra $\mathcal{B}_\sigma(\mathcal{R})$ of $\mathfrak{F}_e$ satisfying \eqref{eq:main}. Let $(\Omega, \Sigma)$ be a measurable space with $\Sigma$ Boolean isomorphic to $\mathcal{B}_\sigma(\mathcal{R})$ by means of a Boolean isomorphism $\varphi: \mathcal{B}_\sigma(\mathcal{R}) \to \Sigma$ (see the Stone Theorem, \cite[Theorem~7.11]{Jec}). Let $\mu$ be the measure on $\Sigma$ defined by $\mu(A) = \mu' \bigl( \varphi^{-1}(A) \bigr)$ for all $A \in \Sigma$. Now we show that $(\Omega, \Sigma, \mu)$ is the desired probability measure space. Observe that the system $(\widetilde{r}_i)_{i \in I}$ in $L_0(\mu)$ defined by
$$
\widetilde{r}_i = \mathbf{1}_{\varphi(r_i^+)} - \mathbf{1}_{\varphi(r_i^-)}, \,\,\, i \in I
$$
is a system of independent $\pm 1$-valued random variables with the distribution $\mu\{\widetilde{r}_i = \pm 1\} = 1/2$. It remains to construct a Riesz subspace $F$ of $E$ containing $(r_i)_{i \in I}$ and a Riesz into isomorphism $T: F \to L_0(\mu)$ with $Tr_i = \widetilde{r}_i$ for all $i \in I$. Let $\overline{\mathcal{P}}$ be the particle semialgebra of $\mathfrak{F}_e$ with respect to $(r_i^+)_{i \in I}$. Denote by $F$ the set of all disjoint sums $\bigsqcup_{k=1}^m a_k x_k$, where $m \in \mathbb N$, $a_k \in \mathbb R$ and $x_k \in \overline{\mathcal{P}}$. Then $F$ is a Riesz subspace of $E$ containing $(r_i)_{i \in I}$. We set
$$
T \Bigl( \bigsqcup_{k=1}^m a_k x_k \Bigr) = \bigsqcup_{k=1}^m a_k \mathbf{1}_{\varphi(x_k)} \in L_0(\mu)
$$
for all $\bigsqcup_{k=1}^m a_k x_k \in F$ ($T$ sends a disjoint sum from $F$ to a disjoint sum in $L_0(\mu)$ because $\varphi$ is a Boolean isomorphism). To prove that $T$ has the desired properties is an easy technical exercise.
\end{proof}



\end{document}